\documentclass[preprint,3p,times,sort]{elsarticle}

\usepackage{amsmath,amsfonts,amsthm,amssymb}
\usepackage{graphicx,float}
\usepackage{enumitem,array,multirow,bbm}
\newcommand{\R}{\mathbb{R}}
\newcommand{\C}{\mathbb{C}}

\newcommand{\F}{{\mathcal{F}}}
\newcommand{\vx}{\mathbf{x}}
\newcommand{\rmd}{\mathrm{d}}

\newcommand{\psihn}[1]{\psi^{#1}}

\journal{ }

\usepackage{hyperref}

\usepackage[capitalize,nameinlink,noabbrev]{cleveref}

\theoremstyle{plain}
\newtheorem{theorem}{Theorem}[section]
\newtheorem{lemma}[theorem]{Lemma}
\newtheorem{proposition}[theorem]{Proposition}

\theoremstyle{definition}

\theoremstyle{remark}
\newtheorem{remark}[theorem]{Remark}

\numberwithin{equation}{section}
\numberwithin{figure}{section}
\numberwithin{table}{section}

\crefname{exmp}{Example}{Examples}
\crefname{hypothesis}{Hypothesis}{Hypotheses}
\crefname{conj}{Conjecture}{Conjectures}
\crefformat{equation}{\textup{#2(#1)#3}}
\crefrangeformat{equation}{\textup{#3(#1)#4--#5(#2)#6}}
\crefmultiformat{equation}{\textup{#2(#1)#3}}{ and \textup{#2(#1)#3}}
{, \textup{#2(#1)#3}}{, and \textup{#2(#1)#3}}
\crefrangemultiformat{equation}{\textup{#3(#1)#4--#5(#2)#6}}%
{ and \textup{#3(#1)#4--#5(#2)#6}}{, \textup{#3(#1)#4--#5(#2)#6}}{, and \textup{#3(#1)#4--#5(#2)#6}}

\Crefformat{equation}{#2Equation~\textup{(#1)}#3}
\Crefrangeformat{equation}{Equations~\textup{#3(#1)#4--#5(#2)#6}}
\Crefmultiformat{equation}{Equations~\textup{#2(#1)#3}}{ and \textup{#2(#1)#3}}
{, \textup{#2(#1)#3}}{, and \textup{#2(#1)#3}}
\Crefrangemultiformat{equation}{Equations~\textup{#3(#1)#4--#5(#2)#6}}%
{ and \textup{#3(#1)#4--#5(#2)#6}}{, \textup{#3(#1)#4--#5(#2)#6}}{, and \textup{#3(#1)#4--#5(#2)#6}}

\begin{document}
	
	\begin{frontmatter}
		
		
		
		\title{A Lawson-time-splitting extended Fourier pseudospectral method for the Gross-Pitaevskii equation with time-dependent low regularity potential}
		
		\author[label1]{Bo Lin}\ead{linbo@u.nus.edu}
		\author[label2]{Ying Ma}\ead{maying@bjut.edu.cn}
		\author[label1]{Chushan Wang}\ead{e0546091@u.nus.edu}
		\affiliation[label1]{organization={Department of Mathematics, National University of Singapore},
			city={Singapore},
			postcode={119076},
			country={Singapore}}
		
		\affiliation[label2]{organization={Department of Mathematics, School of Mathematics, Statistics and Mechanics, Beijing University of Technology},
			city={Beijing},
			postcode={100124},
			country={China}}
		
		%
		
		\begin{abstract}
			We propose a Lawson-time-splitting extended Fourier pseudospectral (LTSeFP) method for the numerical integration of the Gross-Pitaevskii equation with time-dependent potential that is of low regularity in space. For the spatial discretization of low regularity potential, we use an extended Fourier pseudospectral (eFP) method, i.e., we compute the discrete Fourier transform of the low regularity potential in an extended window. For the temporal discretization, to efficiently implement the eFP method for time-dependent low regularity potential, we combine the standard time-splitting method with a Lawson-type exponential integrator to integrate potential and nonlinearity differently. The LTSeFP method is both accurate and efficient: it achieves first-order convergence in time and optimal-order convergence in space in $L^2$-norm under low regularity potential, while the computational cost is comparable to the standard time-splitting Fourier pseudospectral method. Theoretically, we also prove such convergence orders for a large class of spatially low regularity time-dependent potential. Extensive numerical results are reported to confirm the error estimates and to demonstrate the superiority of our method. 
		\end{abstract}
		
		
		
		
	\end{frontmatter}
	
	
	\section{Introduction}
	In this paper, we consider the following Gross-Pitaevskii equation (GPE) with time-dependent potential as
	\begin{equation}\label{NLSE}
		\left\{
		\begin{aligned}
			&i \partial_t \psi(\vx, t) = -\Delta \psi(\vx, t) + V(\vx, t) \psi(\vx, t) + \beta |\psi(\vx, t)|^2 \psi(\vx, t), && \vx \in \Omega, \  t>0, \\
			&\psi(\vx, 0) = \psi_0(\vx), && \vx \in \overline{\Omega},
		\end{aligned}
		\right.
	\end{equation}
	where $t$ is time, $\vx \in \R^d$ is the spatial coordinate, $ \psi:=\psi(\vx, t) $ is a complex-valued wave function, $ V := V(\vx, t) $ is a given real-valued time-dependent external potential, and $\beta \in \R$ is some given constant. Here we are particularly interested in $V$ that is of low regularity in space, i.e., $V(\cdot, t) \in L^\infty(\Omega)$ for $t \geq 0$. 
	
	The Gross-Pitaevskii equation, which is also known as the cubic nonlinear Schr\"odinger equation (NLSE), is widely adopted in physics applications such as the modeling and simulation of Bose-Einstein condensation (BEC), nonlinear optics and plasma physics, and we refer to \cite{ESY,review_2013,NLS,NLSopt,Ant,bao2003JCP} for the mathematical theory and numerical methods. Due to different physical settings, the potential $V$ may take different forms. In particular, many of them may be of low regularity in space, such as the square-well or step potential where $V$ is discontinuous in space; the triangular potential or the power law potential where $V$ has spatially discontinuous or unbounded derivatives; or the random or disorder potential which is very rough in space. For the applications involving the GPE with low regularity potential, we refer to \cite{bao2023_EWI,bao2023_eFP} and references therein for the details. 
	
	To solve the GPE \cref{NLSE} with the aforementioned low regularity time-independent potential, many works have been done recently. In \cite{henning2017}, the Crank-Nicolson Galerkin method is analyzed, which is the first error estimate for the GPE with purely $L^\infty$-potential. Later, some low regularity integrators are designed to reduce the regularity requirements on potential and the exact solution at the same time \cite{zhao2021,bronsard2022}. Recently, some Gautschi-type exponential wave integrators (EWI) were designed and proved to be of optimal convergence orders in the presence of low regularity potential \cite{bao2023_EWI,bao2023_sEWI}. Similar optimal error bounds of time-splitting methods are also established under a CFL-type time step size restriction \cite{bao2023_improved,bao2023_eFP}. Here, we also mention some related works that consider the temporally highly oscillatory potential \cite{su2020} and that account for low regularity nonlinearity \cite{bao2023_semi_smooth,bao2023_EWI,bao2023_sEWI,bao2023_improved}. 
	
	However, all the works mentioned above have primarily focused on low regularity potential that is static. In many experimental settings, it is necessary to take into account the time-dependency of potential or one wants to manipulate the external potential dynamically. For example, the amplitude modulation of trapping potential is used to study the laser cooling technique \cite{PRAampmod}. In the experimental study of BEC, the narrow potential barrier or well is applied by a focused laser beam, whose power gradually increases from zero at the start to a maximum value, which is found to be essential to the observation of expected phenomena \cite{Atdelta,PRAmovobsbox}. In addition to the time dependency of amplitudes, many applications require the incorporation of a moving potential. Typical examples include in the BEC setting a moving obstacle in superfluid, which will cause excitations and lead to fruitful phenomena such as quantized vertices \cite{PRLmovingobstacle} or self-interacting matter wave patterns \cite{PRAmovobsbox}. A Similar moving potential is also considered in quantum transport \cite{PRAmovDelta,PRAmovDeltaQD,Transportboth}. Actually, in many applications including those mentioned above, the potential is time-dependent in both amplitude and location. 
	
	Considering such wide applications, the development of accurate and efficient numerical methods with rigorous error analysis for the GPE with time-dependent spatially low regularity potential is paramount. In this work, we design a Lawson-time-splitting extended Fourier pseudospectral (LTSeFP) method to solve the GPE \cref{NLSE} with time-dependent low regularity potential, which generalizes existing works on time-independent potential. We also rigorously prove that the LTSeFP method converges at first order in time and at optimal order with respect to the regularity of the exact solution in space for a large class of time-dependent spatially $L^\infty$ potential. In the following, we briefly explain the idea of the design and analysis of this method. 
	
	For spatial discretization, the extended Fourier pseudospectral (eFP) method is adopted due to spatially low regularity potential, which can provide optimal convergence orders as shown in \cite{bao2023_eFP}. To implement the eFP method together with the standard time-splitting method in time, one needs to precompute ``exact" Fourier coefficients of the exponential function of potential. When the potential is time-dependent, this preprocessing needs to be done for low regularity potential at all time steps. Although it is a one-time work for each given potential and a fixed temporal mesh, the computational cost of this preprocessing could still be expensive, potentially constraining its practical applications. In fact, for a majority of physically relevant potential, the time dependency lies only in the amplitude and location. Leveraging this aspect, we propose a modification of the time-splitting method: expand the exponential function of low regularity potential to first order. With such modification, the Fourier coefficients of the function modulated in amplitude and location can be easily obtained as some multiple or phase shift of Fourier coefficients of the original function. 
	Then the workload of preprocessing is significantly reduced to computing the Fourier coefficients of a time-independent function, which can thus be neglected. Consequently, the computational efficiency of the LTSeFP method is comparable to that of the highly efficient standard time-splitting Fourier pseudospectral method. However, the accuracy of the LTSeFP method is significantly higher than the time-splitting Fourier pseudospectral method in the presence of spatially low regularity potential. 
	
	In terms of the error estimates with spatially low regularity potential, we use the regularity compensation oscillation (RCO) technique, which was originally proposed in \cite{RCO_SE,RCO_KG} to analyze long-time errors and generalized in \cite{bao2023_improved,bao2023_sEWI} to obtain error bounds for time-independent low regularity potential. By RCO, we estimate high-frequency errors by the regularity of the exact solution and analyze the error cancellation in the accumulation of low-frequency errors. As a result, the regularity requirement on potential is weakened from second-order spatially differentiable to first-order temporally differentiable, which allows us to establish error estimates for spatially $L^\infty$ potential. In the proof, a CFL-type time step size restriction is imposed which is necessary and optimal as justified by extensive numerical results. 
	
	The rest of the paper is organized as follows. In \cref{sec:2}, we present the LTSeFP method and explain some implementation issues. \cref{sec:3} is devoted to the error estimates of the LTSeFP method. Numerical results are reported in \cref{sec:4}. Finally, some conclusions are drawn in \cref{sec:5}. Throughout the paper, we adopt standard notations of Sobolev spaces as well as their corresponding norms. We denote by $ C $ a generic positive constant independent of the mesh size $ h $ and time step size $ \tau $, and by $ C(\alpha) $ a generic positive constant depending only on the parameter $ \alpha $. The notation $ A \lesssim B $ is used to represent that there exists a generic constant $ C>0 $, such that $ |A| \leq CB $. 
	
	\section{A Lawson-time-splitting extended Fourier pseudospectral method}\label{sec:2}
	In this section, we introduce the Lawson-time-splitting extended Fourier pseudospectral method (LTSeFP) to solve the NLSE with time-dependent low regularity potential. For simplicity of the presentation and to avoid heavy notations, we only present the numerical schemes in one dimension (1D) and take $ \Omega = (a, b) $. Generalizations to two dimensions (2D) and three dimensions (3D) are straightforward and the error estimates remain unchanged. We define periodic Sobolev spaces as (see, e.g. \cite{bronsard2022}, for the definition in phase space)
	\begin{equation*}
		H_\text{per}^m(\Omega) := \{\phi \in H^m(\Omega) : \phi^{(k)}(a) = \phi^{(k)}(b), \ k=0, \cdots, m-1\}, \quad m \geq 1. 
	\end{equation*}
	
	\subsection{Semidiscretization in time}
	Based on the operator splitting techniques, the NLSE \cref{NLSE} can be decomposed into two subproblems. The first one is  
	\begin{equation}
		\left\{
		\begin{aligned}
			&\partial_t \psi(x, t+s) = i \Delta \psi(x, t+s), \quad x \in \Omega, \quad s>0, \\
			&\psi(x, t) = \psi_0(x), \quad x \in \overline{\Omega},
		\end{aligned}
		\right.
	\end{equation}
	which can be formally integrated exactly in time as
	\begin{equation}\label{eq:linear_step}
		\psi(\cdot, t+s) = e^{i s \Delta} \psi_0(\cdot), \qquad s \geq 0.
	\end{equation}
	The second one is
	\begin{equation}
		\left\{
		\begin{aligned}
			&\partial_t \psi(x, t+s) = -iV(x, t+s)\psi(x, t+s) -i f(|\psi(x, t+s)|^2)\psi(x, t+s), \quad x \in \Omega, \quad  s>0, \\
			&\psi(x, t) = \psi_0(x), \quad x \in \overline{\Omega},
		\end{aligned}
		\right.
	\end{equation}
	which, by noting
	$|\psi(x, t)|=|\psi_0(x)|$ for $t \geq 0$, can be integrated exactly in time as
	\begin{equation}\label{eq:phi_B_tilde_def}
		\psi(x, t+s) = {\widetilde \Phi_B}^{t, s} (\psi_0)(x) := \psi_0(x)e^{- i \int_0^s V(x, t+\sigma) \rmd \sigma - i s f(|\psi_0(x)|^2))}, \quad x \in \overline{\Omega}, \quad s \geq 0. 
	\end{equation}
	Here, we make a further approximation to the nonlinear step ${\widetilde \Phi_B}^{t, s} (\phi) $ as 
	\begin{equation}\label{eq:approximation}
		{\widetilde \Phi_B}^{t, s} (\phi)(x) = e^{- i \int_0^s V(x, t+s) \rmd s} \phi(x) e^{- i s f(|\phi(x)|^2)} \approx \left( 1 - i \tau V(x, t) \right) \left( \phi(x) e^{- i s f(|\phi(x)|^2)} \right) =: {\widetilde \varphi_B}^{t, s} (\phi)(x).   
	\end{equation}
	As we shall show in the following, with this modified approximation to potential, the computational cost of the preprocessing for the full discretization scheme can be greatly reduced. 
	
	Choose a time step size  $ \tau > 0 $, denote time steps as $ t_n = n \tau $ for $ n = 0, 1, ... $, and let $ \psi^{[n]}(\cdot) $ be the approximation of $ \psi(\cdot, t_n) $ for $ n \geq 0 $. Then a Lawson-time-splitting (LTS) approximation of the NLSE \cref{NLSE} based on the Lie-Trotter splitting scheme is given as
	\begin{equation}\label{eq:LT}
		\psi^{[n+1]} = e^{i \tau \Delta} {\widetilde \varphi}_{B}^{t_n, \tau}\left (\psi^{[n]} \right ), \quad n \geq 0, 
	\end{equation}
	with $ \psi^{[0]}(x) = \psi_0(x) $ for $x\in\overline{\Omega}$. 
	
	\begin{remark}\label{rem:Lawson}
		When $f(\rho) \equiv 0$ in the GPE \cref{NLSE}, the LTS method collapses to the first-order Lawson integrator \cite{ExpInt,feng2022lawson}. When $V(x) \equiv 0$ in \cref{NLSE}, the LTS method collapses to the standard Lie-Trotter time-splitting method \cite{su2020,schratz2016}. 
	\end{remark}
	
	\subsection{Full discretization}
	We further discretize the semi-discretization \eqref{eq:LT} in space to obtain a fully discrete scheme. Choose a mesh size $h = (b-a)/N$ with $N$ being a positive even integer, and denote the grid points as
	\begin{equation}
		x_j = a + jh, \qquad j \in \mathcal{T}_N^0 := \{0, 1, \cdots, N\}. 
	\end{equation}
	Define the index set of frequencies as
	\begin{equation}
		\mathcal{T}_N = \left\{-\frac{N}{2}, \cdots, \frac{N}{2}-1 \right\}. 
	\end{equation}
	Denote
	\begin{align}
		&X_N = \text{span}\left\{e^{i \mu_l(x - a)}: l \in \mathcal{T}_N\right\}, \quad \mu_l = \frac{2 \pi l}{b-a}, \\
		&Y_N = \{v=(v_0, v_1, \cdots, v_N)^T \in \C^{N+1}: v_0 = v_N\}. 
	\end{align}
	Let $ P_N:L^2(\Omega) \rightarrow X_N $ be the standard $ L^2 $-projection onto $ X_N $ and $I_N: Y_N \rightarrow X_N$ be the standard Fourier interpolation operator as 
	\begin{equation}
		(P_N u)(x) = \sum_{l \in \mathcal{T}_N} \widehat u_l e^{i \mu_l(x - a)}, \quad (I_N v)(x) = \sum_{l \in \mathcal{T}_N} \widetilde{v}_l e^{i \mu_l(x - a)},  \qquad x \in \overline{\Omega} = [a, b], 
	\end{equation}
	where $ u \in L^2(\Omega) $, $v \in Y_N$ and
	\begin{equation}\label{eq:hat}
		\widehat{u}_l = \frac{1}{b-a} \int_a^b u(x) e^{-i \mu_l (x-a)} \rmd x, \quad \widetilde{v}_l = \frac{1}{N} \sum_{j=0}^{N-1} v_j e^{- i \mu_l (x_j - a)}, \qquad l \in \mathcal{T}_N. 
	\end{equation}
	The interpolation operator $I_N$ can be similarly defined on $C_\text{per}(\overline{\Omega}) := \{\phi \in C(\overline{\Omega}): \phi(a) = \phi(b)\}$ by identifying $\phi \in C_\text{per}(\overline{\Omega})$ with $ v = (v_0, \cdots, v_N)^T \in Y_N$ satisfying $v_j  = \phi(x_j)$ for $j \in \mathcal{T}_N^0$.  
	
	Let $\psihn{n}_j$ be the numerical approximation to $\psi(x_j, t_n)$ for $j \in \mathcal{T}_N^0$ and $n \geq 0$, and denote $\psihn{n}: = (\psihn{n}_0, \psihn{n}_1, \cdots, \psihn{n}_{N})^T \in Y_N$. Then the Lawson-time-splitting extended Fourier pseudospectral (LTSeFP) method reads 
	\begin{equation}\label{LTEFP_scheme}
		\begin{aligned}
			\begin{aligned}
				&\psi^{(1)}(x) =  \left( 1 - i \tau V(x, t_n) \right) I_N \left( \psihn{n} e^{-i \tau f(|\psihn{n}|^2)} \right)(x), \quad x \in \Omega \\
				&\psihn{n+1}_j = \sum_{l \in \mathcal{T}_N} e^{- i \tau \mu_l^2} \widehat{(\psi^{(1)})}_l e^{i \mu_l(x_j-a)}, \quad j \in \mathcal{T}_N^0, \quad n \geq 0, 
			\end{aligned} 
		\end{aligned}
	\end{equation}
	where $\psihn{0}_j = \psi_0(x_j)$ for $j \in \mathcal{T}_N^0$. 
	
	\begin{remark}[Fourier pseudospectral method]\label{rem:FP}
		For comparison, we also present here the standard time-splitting Fourier pseudospectral (TSFP) method
		\begin{equation*}\label{TSFP_scheme}
			\begin{aligned}
				\begin{aligned}
					&\psi^{(1)}_j =  \psihn{n}_j e^{-i \tau (V(x_j, t_n) + f(|\psihn{n}_j|^2))}, \quad j \in \mathcal{T}_N^0, \\
					&\psihn{n+1}_j = \sum_{l \in \mathcal{T}_N} e^{- i \tau \mu_l^2} \widetilde{(\psi^{(1)})}_l e^{i \mu_l(x_j-a)}, \quad j \in \mathcal{T}_N^0, \quad n \geq 0, 
				\end{aligned} 
			\end{aligned}
		\end{equation*}
		and the Lawson-time-splitting Fourier pseudospectral (LTSFP) method
		\begin{equation*}\label{LTSFP_scheme}
			\begin{aligned}
				\begin{aligned}
					&\psi^{(1)}_j =  \left( 1 - i \tau V(x_j, t_n) \right) \left( \psihn{n}_j e^{-i \tau f(|\psihn{n}_j|^2)} \right), \quad j \in \mathcal{T}_N^0, \\
					&\psihn{n+1}_j = \sum_{l \in \mathcal{T}_N} e^{- i \tau \mu_l^2} \widetilde{(\psi^{(1)})}_l e^{i \mu_l(x_j-a)}, \quad j \in \mathcal{T}_N^0, \quad n \geq 0, 
				\end{aligned} 
			\end{aligned}
		\end{equation*}
		where $\psihn{0}_j = \psi_0(x_j)$ for $j \in \mathcal{T}_N^0$ in both schemes. 
	\end{remark}
	
	To implement the LTSeFP scheme \cref{LTEFP_scheme}, one needs to compute, at the $n$-th time step, $N$ Fourier coefficients of $\psi^{(1)}$. Then the main effort is to compute the Fourier coefficients of
	\begin{equation}
		V (\cdot, t_n)  I_N \left( \phi e^{- i \tau f(|\phi|^2)} \right)(\cdot), \quad \phi \in X_N, 
	\end{equation}
	which can be obtained by using FFT for an extended vector as shown below. Let $w=I_N \left( \phi e^{- i \tau f(|\phi|^2)} \right)$. Noting that $ w \in X_N$ and that $I_{4N}$ is an identity on $X_{4N}$, we have
	\begin{equation}
		P_N \left( V (\cdot, t_n) w \right) = P_N \left( P_{2N} (V (\cdot, t_n)) w \right) = P_N I_{4N} \left( P_{2N} (V (\cdot, t_n)) w \right), 
	\end{equation}
	which implies
	\begin{equation}\label{eq:Vw_hat}
		\widehat{\left( V (\cdot, t_n) w \right)}_l = \widetilde{G}_l, \quad l \in \mathcal{T}_N,   
	\end{equation}
	where $G \in Y_{4N}$ and 
	\begin{equation}\label{eq:G_def}
		G_j = P_{2N} (V (\cdot, t_n)) (y_j) \times w(y_j), \quad y_j = a + j \frac{b-a}{4N}, \quad  j \in \mathcal{T}^0_{4N}. 
	\end{equation}
	In terms of \cref{eq:Vw_hat,eq:G_def}, to efficiently implement the LTSeFP method, one can precompute the Fourier coefficients of $V$ at each $t_n$, i.e.,
	\begin{equation}\label{eq:V_hat_eFP}
		\widehat{V(\cdot, t_n)_l}, \quad l \in \mathcal{T}_{2N}, \quad  n \geq 0. 
	\end{equation}
	On one hand, for a given time-dependent potential $V$ and a fixed temporal mesh, this only needs to be done once in practice. On the other hand, for a large class of physically relevant time-dependent potential, due to the use of Lawson-type integrator for the temporal discretization of the potential function (i.e., the approximation \cref{eq:approximation}) in the LTSeFP scheme, the computation of \cref{eq:V_hat_eFP} can be done very efficiently as we shall present in the following.  
	
	Usually, the time-dependent potential $V(x, t)$ that is of physics interest can be summarized as (the composition of) the following two typical forms:
	\begin{enumerate}[label = (\roman*)]
		\item Space-time separable potential: 
		\begin{equation}\label{eq:V_sep}
			V(x, t) = A(t)V_0(x), \quad x \in \Omega, \quad t \geq 0. 
		\end{equation}
		
		\item Moving potential: (recalling that $\Omega$ is equipped with periodic boundary condition)
		\begin{equation}\label{eq:V_trans}
			V(x, t) = V_0(x + \alpha(t)), \quad x \in \Omega, \quad t \geq 0. 
		\end{equation}
		
	\end{enumerate}
	For these two types of potential, the computation of the Fourier coefficients of $V(\cdot, t)$ at different time $t$ can be reduced to the computation of the Fourier coefficients of $V_0$, which is time-independent. In particular, for the space-time separable potential \cref{eq:V_sep}, we have
	\begin{equation}
		\widehat{V(\cdot, t)_l} = A(t) \widehat{(V_0)_l}, \quad l \in \mathcal{T}_{2N}, \quad t \geq 0, 
	\end{equation}
	and for the moving potential \cref{eq:V_trans}, we have
	\begin{equation}
		\widehat{V(\cdot, t)_l} = e^{i \mu_l \alpha(t)} \widehat{(V_0)_l}, \quad l \in \mathcal{T}_{2N}, \quad t \geq 0. 
	\end{equation}
	
	As a result, as long as one precomputed the Fourier coefficients of a time-independent potential $V_0$, the LTSeFP method can be directly implemented with various time-dependent potentials given by \cref{eq:V_sep} or \cref{eq:V_trans} without any additional computational cost for the preprocessing. 
	
	Thus, the LTSeFP method is as efficient as the standard time-splitting Fourier pseudospectral method or the Lawson-time-splitting Fourier pseudospectral method shown in \cref{rem:FP}. However, the spatial convergence order of the LTSeFP method is optimal with respect to the regularity of the exact solution: for $H^m$-solution, the spatial error is at $O(h^m)$ in $L^2$-norm. In other words, the eFP method can always provide optimal spatial convergence orders in practical computation which is not the case for the Fourier pseudospectral method as we shall present in \cref{sec:4}. 
	
	\begin{remark}
		Sometimes, the time-dependent potential can be separated into two parts $ V(x, t) = V_1(x, t) + V_2(x, t) $ where $V_1$ is spatially low regularity while $V_2$ is spatially sufficiently smooth (i.e., $V_2(\cdot, t)$ has at least the same regularity as the exact solution $\psi(\cdot, t)$). In such case, one only needs to apply the eFP method to the low regularity part of potential and the corresponding full discretization scheme \cref{LTEFP_scheme} can be modified as
		\begin{equation*}
				\begin{aligned}
					&\psi^{(1)}(x) =  \left( 1 - i \tau V_1(x, t_n) \right) I_N \left( \psihn{n} e^{-i \tau (V_2^n + f(|\psihn{n}|^2))} \right)(x), \quad x \in \Omega, \\
					&\psihn{n+1}_j = \sum_{l \in \mathcal{T}_N} e^{- i \tau \mu_l^2} \widehat{(\psi^{(1)})}_l e^{i \mu_l(x_j-a)}, \quad j \in \mathcal{T}_N^0, \quad n \geq 0, 
				\end{aligned} 
		\end{equation*}
		where $V_2^n = (V_2(x_0, t_n), \cdots, V_2(x_N, t_n))^T \in Y_N$ and $\psihn{0}_j = \psi_0(x_j)$ for $j \in \mathcal{T}_N^0$. 
	\end{remark}
	
	\section{Error estimates of the LTSeFP method}\label{sec:3}
	In this section, we shall prove the error estimate of the LTSeFP method for the GPE \cref{NLSE} with spatially low regularity potential. Before presenting the main results, we first state our assumptions on potential and the exact solution. Let $0<T<T_\text{max}$ with $T_\text{max}>0$ being the maximum existence time of the exact solution of the GPE \cref{NLSE}. For time-dependent low regularity potential, we assume that 
	\begin{equation}\label{A1}
		V \in C([0, T]; L^\infty(\Omega)) \cap C^1([0, T]; L^2(\Omega)). 
	\end{equation}
	For the exact solution, we assume that 
	\begin{equation}\label{A2}
		\psi \in C([0, T]; H^m_\text{per}(\Omega)) \cap C^1([0, T]; L^2(\Omega)), \quad m \geq 2. 
	\end{equation}
	\begin{theorem}\label{thm:main}
		Under the assumptions \cref{A1,A2}, there exists $h_0>0$ small enough such that when $0<h<h_0$ and $\tau \leq h^2/\pi$, we have
		\begin{equation*}
			\| \psi(\cdot, t_n) - \psihn{n} \|_{L^2} \lesssim \tau + h^m, \quad \| \psi(\cdot, t_n) - \psihn{n} \|_{H^1} \lesssim \tau^\frac{1}{2} + h^{m-1}, \quad 0 \leq n \leq T/\tau. 
		\end{equation*}
	\end{theorem}
	
	\begin{remark}
		Compared to the standard error estimates of time-splitting methods for a time-dependent spatially smooth potential (e.g., \cite{su2020}), we require an additional $C^1$ temporal regularity of $V$ in \cref{A1}. This is essentially due to the low regularity of $V$ in space instead of the modification by the Lawson-type exponential integrator. In fact, such additional regularity in time is also needed to establish $H^2$ well-posedness of the GPE \cref{NLSE} with spatially low regularity potential, where one shall differentiate the equation once in time rather than twice in space, which thus requires $V$ to be at least once differentiable in time (see Chapter 4.8 of \cite{cazenave2003} or \cite{kato1987}). In this aspect, when $V$ is of low regularity in space (e.g., $V(\cdot, t) \in L^\infty(\Omega)$), for \cref{A2} to hold, it seems the assumption \cref{A1} is nearly optimal and cannot be essentially relaxed. 
	\end{remark}
	
	For simplicity of the presentation, we define a constant
	\begin{equation*}
		M := \max \left\{ \| V \|_{L^\infty([0, T]; L^\infty(\Omega))}, \| \partial_t V \|_{L^\infty([0, T]; L^2(\Omega))}, \| \psi \|_{L^\infty([0, T]; H^m(\Omega))}, \| \partial_t \psi \|_{L^\infty([0, T]; L^2(\Omega))} \right\}. 
	\end{equation*}
	Moreover, we set 
	\begin{equation}
		f(\rho) = \beta \rho, \quad \rho \geq 0,   
	\end{equation}
	and define an operator
	\begin{equation}
		\F(v)(x) = \F(v(x)) := f(|v(x)|^2)v(x), \quad v \in L^2(\Omega). 
	\end{equation}
	Also, we shall abbreviate $\psi(\cdot, t)$ as $\psi(t)$ and $V(\cdot, t)$ as $V(t)$. 
	
	To carry out the error analysis, we define the numerical integrator as 
	\begin{align}
		{\Phi}_{B}^{t_n, \tau} \left( \phi \right)
		&:= P_N \left[ \left( 1 - i \tau V(t_n)  \right) I_N \left( \phi e^{- i \tau f(|\phi|^2)} \right) \right] \notag \\
		&= I_N \left( \phi e^{- i \tau f(|\phi|^2)} \right) - i \tau P_N \left[ V (t_n)  I_N \left( \phi e^{- i \tau f(|\phi|^2)} \right) \right], \quad \phi \in X_N.  \label{eq:phi_B_def}
	\end{align}
	Then, for $\psihn{n} \ (0 \leq n \leq T/\tau)$ obtained from \cref{LTEFP_scheme}, we have
	\begin{equation}\label{eq:phi_def}
		I_N \psihn{n+1} = \Phi^{t_n, \tau}(I_N \psihn{n}) := e^{i\tau\Delta} {\Phi}_{B}^{t_n, \tau} \left( I_N \psihn{n} \right), \quad 0 \leq n \leq T/\tau-1.   
	\end{equation}
	
	We start with the estimate of the local truncation error $\mathcal{E}^n$ defined as
	\begin{equation}\label{eq:En_def}
		\mathcal{E}^n = P_N \psi(t_{n+1}) - \Phi^{t_n, \tau}(P_N \psi(t_n)), \quad 0 \leq n \leq T/\tau-1. 
	\end{equation}
	\begin{lemma}[Lemma 3.1 of \cite{bao2023_eFP}]\label{lem:Hm_bound}
		Let $\phi \in X_N$ and $0 < \tau < 1$. For $m \geq 2$, we have
		\begin{equation*}
			\| \phi (e^{- i \tau f(|\phi|^2)} - 1) \|_{H^m} \leq C(\| \phi \|_{H^m}) \tau.
		\end{equation*}
	\end{lemma}
	
	\begin{proposition}[Error decomposition]\label{prop:local_error_decomp}
		Under the assumptions \cref{A1,A2}, we have
		\begin{equation}
			\mathcal{E}^n = \mathcal{E}_\text{\rm dom}^n + \mathcal{E}_2^n, \quad 0 \leq n \leq T/\tau-1, 
		\end{equation}
		where $\| \mathcal{E}^n_2 \|_{L^2} \lesssim \tau^2 + \tau h^m$ and
		\begin{equation}
			\mathcal{E}_\text{\rm dom}^n = (-\Delta) e^{i\tau\Delta} \int_0^\tau s\int_0^1 e^{-i \theta s \Delta} \rmd \theta P_N \left(V(t_n+s) \psi(t_n + s) + \F(\psi(t_n+s))\right) \rmd s. 
		\end{equation}
	\end{proposition}
	
	\begin{proof}
		By Duhamel's formula, we have
		\begin{align}\label{eq:exact_solution}
			\psi(t_{n+1}) 
			&= e^{i \tau \Delta} \psi(t_n) - i \int_0^\tau e^{i (\tau - s)\Delta} \left(V(t_n+s) \psi(t_n + s) + \F(\psi(t_n+s))\right) \rmd s \notag \\
			&= e^{i \tau \Delta} \psi(t_n) - i e^{i\tau\Delta} \int_0^\tau \left(V(t_n+s) \psi(t_n) + \F(\psi(t_n))\right) \rmd s \notag \\
			&\quad - i e^{i\tau\Delta} \int_0^\tau \left[ V(t_n+s) (\psi(t_n + s) - \psi(t_n)) + (\F(\psi(t_n+s)) - \F(\psi(t_n))) \right] \rmd s \notag \\
			&\quad - i e^{i\tau\Delta} \int_0^\tau (e^{-is\Delta}-I) \left(V(t_n+s) \psi(t_n + s) + \F(\psi(t_n+s))\right) \rmd s. 
		\end{align}
		Recalling \cref{eq:phi_B_def,eq:phi_def}, we get
		\begin{equation}\label{eq:numerical_solution}
			\Phi^{t_n, \tau}(P_N \psi(t_n)) 
			= e^{i\tau\Delta} P_N \left[ \left( 1 - i \tau V(t_n)  \right) I_N \left( \left(P_N \psi(t_n)\right) e^{- i \tau f(|P_N \psi(t_n)|^2)} \right) \right]. 
		\end{equation}
		To measure the error between \cref{eq:exact_solution} and \cref{eq:numerical_solution}, we define an intermediate function
		\begin{equation}\label{eq:W_def}
			W^{n+1} := e^{i\tau\Delta} \widetilde{\Phi}_{B}^{t_n, \tau}(P_N \psi(t_n)), 
		\end{equation}
		where $\widetilde{\Phi}_{B}^{t_n, \tau}$ is defined in \cref{eq:phi_B_tilde_def}. Moreover, define 
		\begin{equation}\label{eq:gamma_def}
			\Gamma^n := \frac{1}{\tau} \int_0^\tau V(t_n + s) \rmd s + f(|P_N \psi(t_n)|^2). 
		\end{equation}
		From \cref{eq:W_def}, recalling \cref{eq:phi_B_tilde_def}, by Taylor's expansion, we obtain
		\begin{align}\label{eq:W}
			W^{n+1} 
			&= e^{i\tau\Delta} P_N \psi(t_n) - i\tau e^{i\tau\Delta} \Gamma^n P_N \psi(t_n) \notag \\
			&\quad -i \tau^2 e^{i\tau\Delta} \int_0^1 (1-\theta) e^{- i \theta \tau \Gamma^n} \rmd \theta \left( \Gamma^n \right)^2 P_N \psi(t_n) . 
		\end{align}
		Subtracting \cref{eq:exact_solution} from \cref{eq:W} and applying $P_N$ on both sides, we have
		\begin{equation}\label{eq:r_def}
			P_N \psi(t_{n+1}) - P_N W^{n+1} = e^n + r^n_1 + r^n_2 + r^n_3, 
		\end{equation}
		where
		\begin{align}
			&e^n = - i e^{i\tau\Delta} P_N \int_0^\tau (e^{-is\Delta}-I) \left(V(t_n+s) \psi(t_n + s) + \F(\psi(t_n+s))\right) \rmd s, \\
			&r^n_1 = - i e^{i\tau\Delta} P_N \int_0^\tau \left(V(t_n+s) \psi(t_n) + \F(\psi(t_n))\right) \rmd s + i\tau e^{i\tau\Delta} P_N \left( \Gamma^n P_N \psi(t_n) \right), \\
			&r^n_2 = - i e^{i\tau\Delta} P_N \int_0^\tau \left[ V(t_n+s) (\psi(t_n + s) - \psi(t_n)) + (\F(\psi(t_n+s)) - \F(\psi(t_n))) \right] \rmd s, \\
			&r^n_3 = -i \tau^2 e^{i\tau\Delta} P_N \int_0^1 (1-\theta) e^{- i \theta \tau \Gamma^n} \rmd \theta \left( \Gamma^n \right)^2 P_N \psi(t_n). 
		\end{align}
		Recalling \cref{eq:gamma_def}, by the boundedness of $e^{i t \Delta}$ and $P_N$ and the standard projection error estimates of $P_N$, we have 
		\begin{align}
			\| r^n_1 \|_{L^2} 
			&\leq \left \| \int_0^\tau \left(V(t_n+s) (\psi(t_n) - P_N \psi(t_n) \right) \rmd s - \tau (\F(\psi(t_n)) - \F(P_N \psi(t_n))) \right \|_{L^2} \notag \\
			&\leq C(M) \tau \| \psi(t_n) - P_N \psi(t_n) \|_{L^2} \lesssim \tau h^m. \label{r1_est}
		\end{align}
		For $r^n_2$, we have, by the boundedness of $e^{i t \Delta}$ and $P_N$, 
		\begin{equation}\label{r2_est}
			\| r^n_2 \|_{L^2} \leq C(M) \tau \| \psi(t_n + s) - \psi(t_n) \|_{L^2} \lesssim \tau^2. 
		\end{equation}
		For $r^n_3$, noting that $\| P_N \psi(t_n) \|_{L^\infty} \lesssim \| P_N \psi(t_n) \|_{H^2} \lesssim \| \psi(t_n) \|_{H^2} \leq C(M)$, we have
		\begin{equation}\label{r3_est}
			\| r^n_3 \|_{L^2} \leq \tau^2 \| \left( \Gamma^n \right)^2 P_N \psi(t_n) \|_{L^2} \lesssim \tau^2 \| \Gamma^n \|_{L^\infty}^2 \| P_N \psi(t_n) \|_{L^\infty} \lesssim \tau^2. 
		\end{equation}
		
		From \cref{eq:W_def,eq:numerical_solution}, recalling \cref{eq:phi_B_tilde_def}, we have
		\begin{align}
			&&P_N W^{n+1} - \Phi^{t_n, \tau}(P_N \psi(t_n)) 
			= e^{i\tau\Delta} P_N \left( e^{- i \int_0^\tau V(t_n+s) \rmd s} (P_N \psi(t_n)) e^{- i \tau f(|P_N \psi(t_n)|^2)} \right) \notag \\
			&&\quad - e^{i\tau\Delta} P_N \left[ \left( 1 - i \tau V(t_n) \right) I_N \left( \left(P_N \psi(t_n)\right) e^{- i \tau f(|P_N \psi(t_n)|^2)} \right) \right].   
		\end{align}
		It follows that, by the boundedness of $e^{it\Delta}$ and $P_N$, 
		\begin{align}\label{eq:est}
			&\| P_N W^{n+1} - \Phi^{t_n, \tau}(P_N \psi(t_n)) \|_{L^2} \notag \\
			&\leq \left\| e^{- i \int_0^\tau V(t_n+s) \rmd s} (I - I_N) \left[(P_N \psi(t_n)) e^{- i \tau f(|P_N \psi(t_n)|^2)} \right] \right\|_{L^2} \notag \\
			&\quad + \left \| \left( e^{- i \int_0^\tau V(t_n+s) \rmd s} - \left( 1 - i \tau V(t_n)  \right) \right) I_N \left( \left(P_N \psi(t_n)\right) e^{- i \tau f(|P_N \psi(t_n)|^2)} \right) \right \|_{L^2} \notag \\
			&\leq \left \| (I - I_N) \left( \left(P_N \psi(t_n)\right) e^{- i \tau f(|P_N \psi(t_n)|^2)} \right) \right \|_{L^2} \notag \\
			&\quad + \left\| e^{- i \int_0^\tau V(t_n+s) \rmd s} - \left( 1 - i \tau V(t_n)  \right) \right\|_{L^2} \left \| I_N \left( \left(P_N \psi(t_n)\right) e^{- i \tau f(|P_N \psi(t_n)|^2)} \right) \right\|_{L^\infty}. 
		\end{align}
		Noting that $I_N$ is an identity on $X_N$, one has, by letting $\phi = P_N \psi(t_n) \in X_N$ and using \cref{lem:Hm_bound},  
		\begin{align}
			&\left \| (I - I_N) \left( \left(P_N \psi(t_n)\right) e^{- i \tau f(|P_N \psi(t_n)|^2)} \right) \right \|_{L^2} \notag \\ 
			&= \left \| (I - I_N) \left( \phi \left(e^{- i \tau f(|\phi|^2)} - 1 \right) \right) \right \|_{L^2} \leq h^m \left \| \phi \left(e^{- i \tau f(|\phi|^2)} - 1 \right) \right \|_{H^m} \lesssim \tau h^m. \label{est_1}
		\end{align}
		By Taylor's expansion, we have
		\begin{align}
			&\left \| e^{- i \int_0^\tau V(t_n+s) \rmd s} - \left( 1 - i \tau V(t_n)  \right) \right \|_{L^2} \notag \\
			&\leq \int_0^\tau \left\| V(t_n+s) - V(t_n) \right \|_{L^2} \rmd s + \int_0^1 (1-\theta) \left \| \int_0^\tau V(t_n+s) \rmd s \right \|_{L^4}^2 \rmd \theta \lesssim \tau^2.  \label{est_2}
		\end{align}
		By Sobolev embedding of $H^2 \hookrightarrow L^\infty$, we have
		\begin{align}
			\left \| I_N \left( \left(P_N \psi(t_n)\right) e^{- i \tau f(|P_N \psi(t_n)|^2)} \right) \right\|_{L^\infty} 
			&\lesssim \left \| I_N \left( \left(P_N \psi(t_n)\right) e^{- i \tau f(|P_N \psi(t_n)|^2)} \right) \right\|_{H^2} \notag \\
			&\lesssim \left \| \left(P_N \psi(t_n)\right) e^{- i \tau f(|P_N \psi(t_n)|^2)} \right\|_{H^2} \leq C(M). \label{est_3}
		\end{align}
		Inserting \cref{est_1,est_2,est_3} into \cref{eq:est} yields
		\begin{equation}
			\| P_N W^{n+1} - \Phi^{t_n, \tau}(P_N \psi(t_n)) \|_{L^2} \lesssim \tau^2 + \tau h^m. 
		\end{equation}
		Recalling \cref{eq:En_def,eq:r_def}, we have
		\begin{equation}
			\mathcal{E}^n = e^n + r^n_1 +  r^n_2 + r^n_3 + P_N W^{n+1} - \Phi^{t_n, \tau}(P_N \psi(t_n)), 
		\end{equation}
		which completes the proof by noting $\mathcal{E}^n_\text{dom} = e^n$ and taking $\mathcal{E}^n_2 = r^n_1 +  r^n_2 + r^n_3 + P_N W^{n+1} - \Phi^{t_n, \tau}(P_N \psi(t_n))$. 
	\end{proof}
	
	To establish the $L^2$-stability estimate, we also need the following estimates whose proof is straightforward and thus is omitted. 
	\begin{lemma}\label{lem:diff_nonl}
		Let $z_1, z_2 \in \C$ such that $|z_1| \leq M_0$ and $|z_2| \leq M_0$. Then we have
		\begin{align}
			&\left| z_1 e^{- i \tau f(|z_1|^2)} - z_2 e^{- i \tau f(|z_2|^2)} \right| \leq (1+C(M_0)\tau) |z_1 - z_2|, \label{eq:diff_1} \quad \\
			&\left| z_1 (e^{- i \tau f(|z_1|^2)} - 1) - z_2 (e^{- i \tau f(|z_2|^2)} - 1) \right| \leq C(M_0) \tau |z_1 - z_2|. \label{eq:diff_2}
		\end{align}
	\end{lemma}
	
	
	\begin{proposition}[$L^\infty$-conditional $L^2$-stability estimate]\label{prop:stability}
		Let $v, w \in X_N$ such that $\| v \|_{L^\infty} \leq M_0$ and $\| w \|_{L^\infty} \leq M_0$. Then we have
		\begin{equation*}
			\left\| \left(\Phi_B^{t_n, \tau} - I\right)(v) - \left(\Phi_B^{t_n, \tau} - I\right)(w) \right\|_{L^2} \leq C(M_0, \| V \|_{L^\infty([0, T]; L^\infty)}) \tau \| v-w \|_{L^2}.   
		\end{equation*}
	\end{proposition}
	
	\begin{proof}
		Recalling \cref{eq:phi_B_def}, we have, for $u \in X_N$, 
		\begin{equation}
			\left(\Phi_B^{t_n, \tau} - I\right)(u) = I_N \left[ u \left(e^{- i \tau f(|u|^2)} - 1 \right) \right] - i \tau P_N \left[ V (t_n)  I_N \left( u e^{- i \tau f(|u|^2)} \right) \right]. 
		\end{equation}
		Noting that $ \| I_N \phi \|_{L^2}^2 = h\sum_{j=0}^{N-1} |\phi(x_j)|^2 $ for $\phi \in C_\text{per}(\overline{\Omega})$, by \cref{lem:diff_nonl} and the boundedness of $P_N$, we can obtain the desired estimate. 
	\end{proof}

	\begin{proof}[Proof of \cref{thm:main}]
		Let $ e^n = P_N \psi(t_n) - I_N \psi^n $ for $ 0 \leq n \leq T/\tau $. By standard projection error estimates of $P_N$, it suffices to establish the estimate of $ e^n $. For $  0 \leq n \leq T/\tau -1 $, we have
		\begin{align}\label{eq:error_eq_improv}
			e^{n+1} 
			&= P_N \psi(t_{n+1}) - I_N \psi^{n+1} \notag\\
			&= P_N \psi(t_{n+1}) - \Phi^{t_n, \tau}(P_N\psi(t_n)) + \Phi^{t_n, \tau}(P_N\psi(t_n)) - \Phi^{t_n, \tau}(I_N \psi^n) \notag \\
			&= e^{i \tau \Delta} e^n + Q^n + \mathcal{E}^n, 
		\end{align}
		where $\mathcal{E}^n$ is defined in \cref{eq:En_def} and 
		\begin{equation}\label{eq:Qn}
			Q^n = e^{i \tau \Delta} \left( \left(\Phi_B^{t_n, \tau}(P_N \psi(t_n)) - P_N \psi(t_n) \right) - \left(\Phi_B^{t_n, \tau}(I_N \psi^n) - I_N \psi^n \right)  \right). 
		\end{equation}
		Iterating \cref{eq:error_eq_improv}, we have
		\begin{equation}\label{eq:error_eq_iter}
			e^{n+1} = e^{i(n+1)\tau\Delta} e^0 + \sum_{k=0}^{n} e^{i (n-k)\tau \Delta} \left( Q^k + \mathcal{E}^k \right), \quad 0 \leq n \leq T/\tau-1. 
		\end{equation}
		For $ Q^n $ in \cref{eq:Qn}, applying \cref{prop:stability} and the isometry property of $e^{it\Delta}$, we have
		\begin{equation*}
			\| Q^n \|_{L^2} \leq C(\| \psi^n \|_{L^\infty}, M) \tau \| e^n \|_{L^2}, \quad 0 \leq n \leq T/\tau - 1, 
		\end{equation*}
		which, together with the isometry property of $e^{i t \Delta}$ and the triangle inequality, implies
		\begin{equation}\label{eq:est_Zn}
			\left \| \sum_{k=0}^{n} e^{i (n-k)\tau \Delta} Q^k \right \|_{L^2} \leq C_1 \tau \sum_{k=0}^{n} \| e^k \|_{L^2}, \quad 0 \leq n \leq T/\tau -1, 
		\end{equation}
		where the constant $C_1$ depends on $\max_{0 \leq k \leq n} \| \psi^k \|_{L^\infty}$ and $M$. 
		By \cref{prop:local_error_decomp} and the isometry property of $e^{i t \Delta}$, we have
		\begin{align}\label{eq:est_En}
			\left\| \sum_{k=0}^{n} e^{i (n-k)\tau \Delta} \mathcal{E}^k \right\|_{L^2} 
			&\lesssim n \tau (\tau + h^m) + \left\| e^{i n \tau \Delta} \mathcal{J}^n \right\|_{L^2}, \quad  0 \leq n \leq T/\tau - 1, 
		\end{align}
		where
		\begin{equation}\label{eq:Jn}
			\mathcal{J}^n := \sum_{k=0}^{n} e^{-ik\tau \Delta} \mathcal{E}_{\rm dom}^k, \quad  0 \leq n \leq T/\tau - 1. 
		\end{equation}
		From \cref{eq:error_eq_iter}, using \cref{eq:est_Zn,eq:est_En}, we have
		\begin{equation}\label{eq:error_eq_LT_L2}
			\| e^{n+1} \|_{L^2} \lesssim \tau + h^m + C_1\tau \sum_{k=0}^n \| e^k \|_{L^2} + \left\| \mathcal{J}^n \right\|_{L^2}, \quad 0 \leq n \leq T/\tau -1. 
		\end{equation}
		For $0 \leq n \leq T/\tau - 1$, set
		\begin{equation}\label{eq:phin}
			\phi^n = e^{i\tau\Delta} \int_0^\tau s\int_0^1 e^{-i \theta s \Delta} \rmd \theta P_N \left(V(t_n+s) \psi(t_n + s) + \F(\psi(t_n+s))\right) \rmd s \in X_N.
		\end{equation}
		By the boundedness of $e^{it\Delta}$ and $P_N$, we have
		\begin{equation}\label{eq:phin_est}
			\| \phi^n \|_{L^2} \leq \tau \int_0^\tau \| V(t_n+s) \psi(t_n + s) + \F(\psi(t_n+s)) \|_{L^2} \rmd s \lesssim \tau^2.  
		\end{equation}
		Define
		\begin{equation}\label{eq:S_def}
			S_{n, l} = \sum_{k=0}^n e^{ i k \tau \mu_l^2}, \qquad l \in \mathcal{T}_N, \quad 0 \leq n \leq T/\tau -1. 
		\end{equation}
		When $ \tau < h^2 / \pi $, recalling $|\sin(x)| \geq 2|x|/\pi$ when $ x \in [0, \pi/2]$, we have
		\begin{equation}\label{eq:S_est}
			|S_{n, l}| = \frac{|1-e^{i (n+1) \tau \mu_l^2}|}{|1-e^{i \tau \mu_l^2}|} \leq \frac{2}{2\sin(\tau \mu_l^2 /2)} \lesssim \frac{1}{\tau \mu_l^2}, \quad 0 \neq l \in \mathcal{T}_N. 
		\end{equation}
		Inserting \cref{eq:phin} into \cref{eq:Jn}, recalling $\mathcal{E}^n_\text{dom} = -\Delta \phi^n$, we have, for $0 \leq n \leq T/\tau-1$, 
		\begin{equation}\label{eq:Jn_phase}
			\widehat {(\mathcal{J}^n)}_{l} = \sum_{k=0}^n e^{ i k \tau \mu_l^2} \widehat{(\mathcal{E}^k_\text{\rm dom})}_l = \sum_{k=0}^n e^{ i k \tau \mu_l^2} \mu_l^2 \widehat{\phi^k_l}, \quad l \in \mathcal{T}_N. 
		\end{equation}
		From \cref{eq:Jn_phase}, using summation by parts and recalling \cref{eq:S_def}, we obtain, for $0 \leq n \leq T/\tau-1$,  
		\begin{equation}\label{eq:summation_by_parts}
			\widehat {(\mathcal{J}^n)}_{l} = S_{n, l} \mu_l^2 \widehat{\phi^n_l} - \sum_{k=0}^{n-1} S_{k, l} \mu_l^2 \left(\widehat{\phi^{k+1}_l} - \widehat{\phi^k_l}\right), \quad l \in \mathcal{T_N}.    
		\end{equation}
		From \cref{eq:summation_by_parts}, by \cref{eq:S_est}, we have, for $0 \leq n \leq T/\tau-1$, 
		\begin{equation}\label{eq:Jn_est}
			\left | \widehat {(\mathcal{J}^n)}_{l} \right | \lesssim \frac{1}{\tau} |\widehat \phi^n_l| + \frac{1}{\tau} \sum_{k=0}^{n-1} \left|\widehat{\phi^{k+1}_l} - \widehat{\phi^k_l} \right|, \quad l \in \mathcal{T}_N. 
		\end{equation}
		From \cref{eq:Jn}, using Parseval's identity, Cauchy inequality, \cref{eq:Jn_phase,eq:Jn_est}, we have, for $0 \leq n \leq T/\tau-1$, 
		\begin{align}\label{eq:est_J_LT_L2}
			\| \mathcal{J}^n \|_{L^2}^2 
			&=(b-a) \sum_{l \in \mathcal{T}_N} \left | \widehat {(\mathcal{J}^n)}_{l} \right |^2 
			\lesssim \frac{1}{\tau^2} \sum_{l \in \mathcal{T}_N} \left|\widehat{\phi^n_l} \right|^2 + \frac{n}{\tau^2} \sum_{k=0}^{n-1} \sum_{l \in \mathcal{T}_N} \left| \widehat{\phi^{k+1}_l} - \widehat{\phi^k_l} \right|^2 \notag\\
			&\lesssim \frac{1}{\tau^2} \| \phi^n \|_{L^2}^2 + \frac{n}{\tau^2} \sum_{k=0}^{n-1} \| \phi^{k+1} - \phi^{k} \|_{L^2}^2. 
		\end{align}
		Recalling \cref{eq:phin}, by the boundedness of $e^{it\Delta}$ and $P_N$, for $0 \leq n \leq T/\tau-2$, we have
		\begin{align}\label{eq:diff_phin_est}
			&\| \phi^{n+1} - \phi^{n} \|_{L^2} \notag \\
			&\leq \tau \int_0^\tau \big( \| V(t_{n+1}+s) \psi(t_{n+1} + s) - V(t_{n}+s) \psi(t_{n} + s) \|_{L^2} \\
			&\hphantom{\leq \tau \int_0^\tau \big(} + \| \F(\psi(t_{n+1}+s)) - \F(\psi(t_{n}+s)) \|_{L^2} \big ) \rmd s \notag \\
			&\leq \tau \int_0^\tau \int_0^\tau \big( \| \partial_\sigma (V(t_n+\sigma + s) \psi(t_n + \sigma + s)) \|_{L^2} \notag \\
			&\hphantom{\leq \tau \int_0^\tau \int_0^\tau \big(} + \| \partial_\sigma \F(\psi(t_{n}+\sigma+s)) \|_{L^2} \big) \rmd \sigma \rmd s \notag \\
			&\lesssim \tau^3 \left(\| \partial_t V \|_{L^\infty([t_n, t_{n+2}]; L^2)} \| \psi \|_{L^\infty([t_n, t_{n+2}]; L^\infty)} + \| V \|_{L^\infty([t_n, t_{n+2}]; L^\infty)} \| \partial_t \psi \|_{L^\infty([t_n, t_{n+2}]; L^2)} \right) \notag \\
			&\quad + \tau^3 \| \psi \|^2_{L^\infty([t_n, t_{n+2}]; L^\infty)} \| \partial_t \psi \|_{L^\infty([t_n, t_{n+2}]; L^2)} \leq C(M) \tau^3. 
		\end{align}
		From \cref{eq:est_J_LT_L2}, using \cref{eq:phin_est,eq:diff_phin_est}, we obtain
		\begin{equation}\label{eq:est_J_final}
			\| \mathcal{J}^n \|^2_{L^2} \lesssim \frac{\tau^4}{\tau^2} + \frac{n^2 \tau^6}{\tau^2} \lesssim \tau^2.   
		\end{equation}
		Inserting \cref{eq:est_J_final} into \cref{eq:error_eq_LT_L2}, we have
		\begin{equation}\label{eq:error_final_L2}
			\| e^{n+1} \|_{L^2} \lesssim \tau + h^m + C_1 \tau \sum_{k=0}^{n} \| e^k \|_{L^2}, 
		\end{equation}
		where $ C_1 $ depends on $M_2$ and $ \max_{0 \leq k \leq n} \| \psi^k \|_{L^\infty} $ and can be controlled by discrete Gronwall's inequality and the standard argument of mathematical induction with the inverse equality $ \| \phi \|_{L^\infty} \leq C_\text{inv} h^{-\frac{d}{2}} \| \phi \|_{L^2} $ for all $\phi  \in X_N $ \cite{book_spectral}. The details can be found \cite{bao2023_improved}. As a result, we obtain
		\begin{equation*}
			\| e^n \|_{L^2} \lesssim \tau + h^m, \quad 0 \leq n \leq T/\tau, 
		\end{equation*}
		which completes the proof of the optimal $L^2$-norm error bound in \cref{thm:main}. The $H^1$-norm error bound follows directly from the $L^2$-norm error bound by the inverse inequality $\| \phi \|_{H^1} \lesssim h^{-1} \| \phi \|_{L^2}$ for $\phi \in X_N$ and the step size restriction $\tau \lesssim h^2$. Thus the proof is completed. 
	\end{proof}
	
	\section{Numerical results}\label{sec:4}
	\subsection{Convergence rates of the LTSeFP method}
	In this section, we shall present some numerical results for different kinds of time-dependent spatially low regularity potential to test the convergence of the LTSeFP method in both time and space. In all the numerical experiments in this subsection, we fix $d=1$, $\Omega = (-16, 16)$, $T = 1$, $\beta = 1$, and choose a Gaussian initial datum
	\begin{equation*}
		\psi_0(x) = e^{-x^2/2}, \quad x \in \Omega. 
	\end{equation*} 
	
	We consider the following typical examples of low regularity potential: 
	\begin{enumerate}[label = (\roman*)]
		\item Static square-well potential: 
		\begin{equation}\label{eq:poten_0}
			V(x, t) = V_0(x), \quad V_0(x) = \left\{
			\begin{aligned}
				&0, &x \in (-4, 4), \\
				&10, &\text{otherwise}, 
			\end{aligned}, 
			\right.
			\qquad x \in \Omega. 
		\end{equation}
		
		\item Space time separable potential:
		\begin{equation}\label{eq:poten_1}
			V(x, t) = \cos(8\pi t) V_1(x), \quad V_1(x) = \left\{
			\begin{aligned}
				&10, &x \in (-2, 2), \\
				&0, &\text{otherwise}, 
			\end{aligned}
			\right.
			,\qquad x \in \Omega. 
		\end{equation} 
		
		
		\item Moving potential barrier: 
		\begin{equation}\label{eq:poten_2}
			V(x, t) = \left\{
			\begin{aligned}
				&0, &x \in (-4, -2-2\sin(8 \pi t)) \cup (2-2\sin(8 \pi t), 4), \\
				&10, &\text{otherwise}, 
			\end{aligned}
			\right.
			, \qquad x \in \Omega. 
		\end{equation}
		One notes that $V(x, t) = V_0(x) + V_1(x+2\sin(8 \pi t)) $ where $V_0$ is given in \cref{eq:poten_0} and $V_1$ is given in \cref{eq:poten_1}. 
		
	\end{enumerate}
	
	The exact solution is computed by the LTSeFP method with $h = h_\text{e}: = 2^{-9}$ and $\tau = \tau_\text{e}:= 10^{-7}$. Note that such choice of $\tau$ and $h$ satisfies $\tau \leq h^2/\pi$. 
	To quantify the error, we introduce the following error functions
	\begin{equation}
		e_{L^2}(t_n) = \| \psi(\cdot, t_n) - I_N \psihn{n} \|_{L^2}, \quad e_{H^1}(t_n) = \| \psi(\cdot, t_n) - I_N \psihn{n} \|_{H^1}. 
	\end{equation}
	
	In the following, we shall report the numerical results for the GPE with the above three potential functions (i) - (iii) to demonstrate the excellent performance of the LTSeFP method. In all the cases, we present convergence results in both time and space. For the temporal convergence, we plot the errors in $L^2$- and $H^1$-norms for different time step sizes $\tau$ under different choices of the mesh size $h$. To show the spatial convergence, we fix the time step size $\tau = \tau_\text{e}$ and compute the errors for different $h$. Moreover, for comparison purposes, we also present the result of using the standard Fourier pseudospectral method to discretize the semi-discrete LTS method \cref{eq:LT}, which is abbreviated as LTSFP method. Assuming that $\psihn{\langle n \rangle} \ (0 \leq n \leq T/\tau)$ is obtained from the LTSFP method, we define the corresponding error functions associated with the LTSFP method as
	\begin{equation}
		\widetilde e_{L^2}(t_n) = \| \psi(\cdot, t_n) - I_N \psihn{\langle n \rangle} \|_{L^2}, \quad \widetilde e_{H^1}(t_n) = \| \psi(\cdot, t_n) - \psihn{\langle n \rangle} \|_{H^1}. 
	\end{equation}
	
	We start with the static well potential \cref{eq:poten_0} and exhibit the temporal errors and spatial errors in \cref{fig:0_time,fig:0_space}, respectively. The markers ``$\ast$" in \cref{fig:0_time} are put at data points satisfying $ \tau \approx h^2/\pi$. We can observe that 
	\begin{itemize}
		\item The temporal convergence of the LTeFP method is first-order in $L^2$-norm and 0.75-order in $H^1$-norm only when the time step size condition $\tau \leq h^2/\pi$ is satisfied, and there is order reduction if $\tau \gg h^2$; 
		\item The spatial convergence of the LTeFP method is 2.5-order in $L^2$-norm and 1.5-order in $H^1$-norm, which are optimal with respect to the regularity of the exact solution which is roughly $H^{2.5}$. In contrast, the convergence rate of the standard Fourier pseudospectral method is only first order in both $L^2$- and $H^1$-norms. 
	\end{itemize}
	These observations validate our error estimates, and show the superiority of extended Fourier pseudospectral method compared to the standard Fourier pseudospectral method in terms of accuracy. In particular, they suggest the time step size restriction $\tau \leq h^2/\pi$ is necessary in practical applications. 
	
	\begin{figure}[htbp]
		\centering
		{\includegraphics[width=0.475\textwidth]{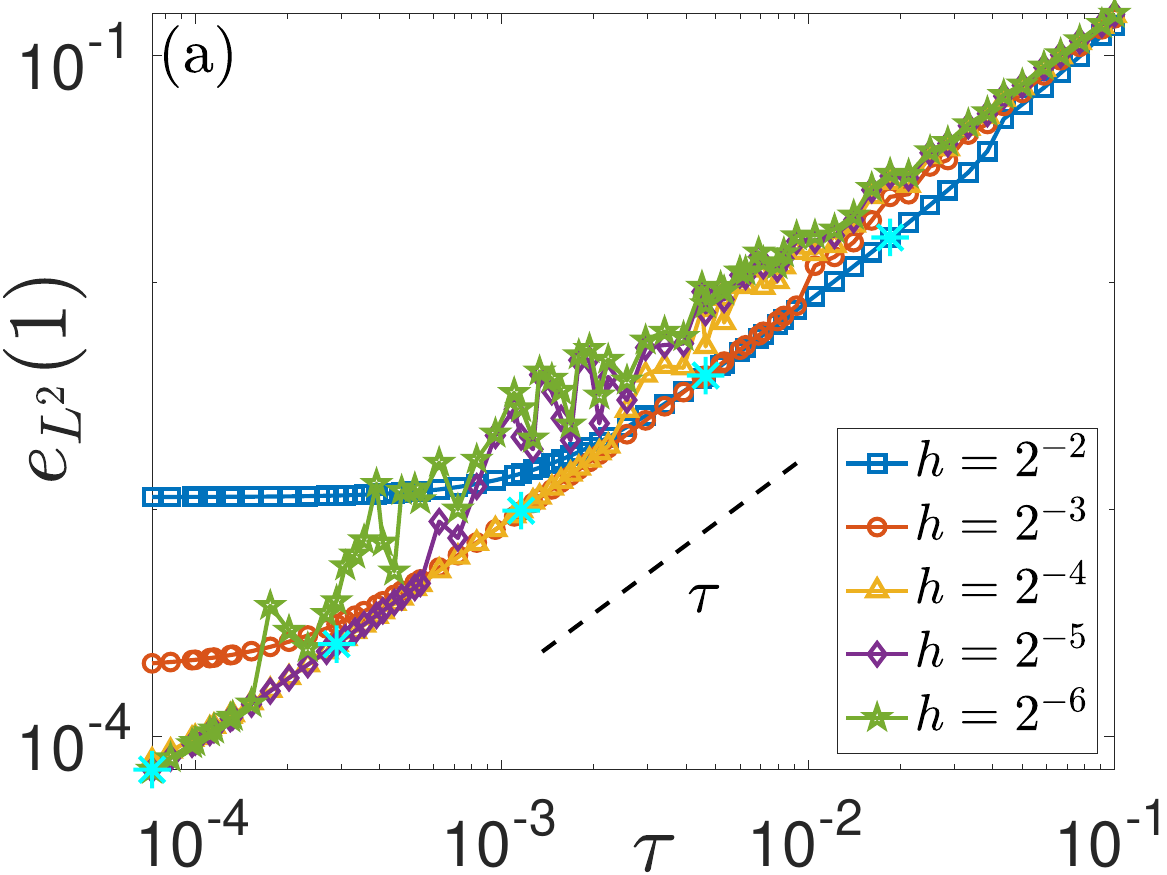}}\hspace{1em}
		{\includegraphics[width=0.475\textwidth]{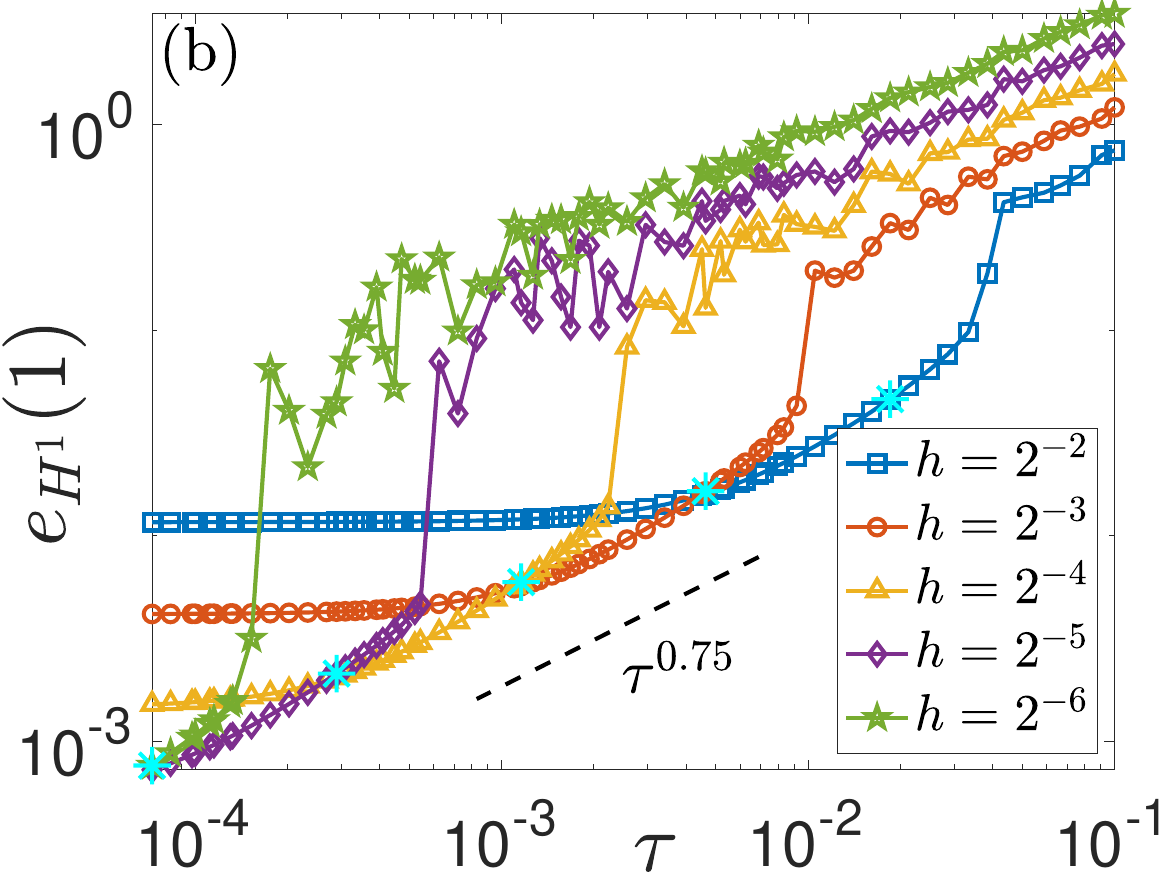}}
		\caption{Temporal errors of the LTSeFP method for the GPE \cref{NLSE} with low regularity potential \cref{eq:poten_0}: (a) $L^2$-norm errors and (b) $H^1$-norm errors }
		\label{fig:0_time}
	\end{figure}
	
	\begin{figure}[htbp]
		\centering
		{\includegraphics[width=0.475\textwidth]{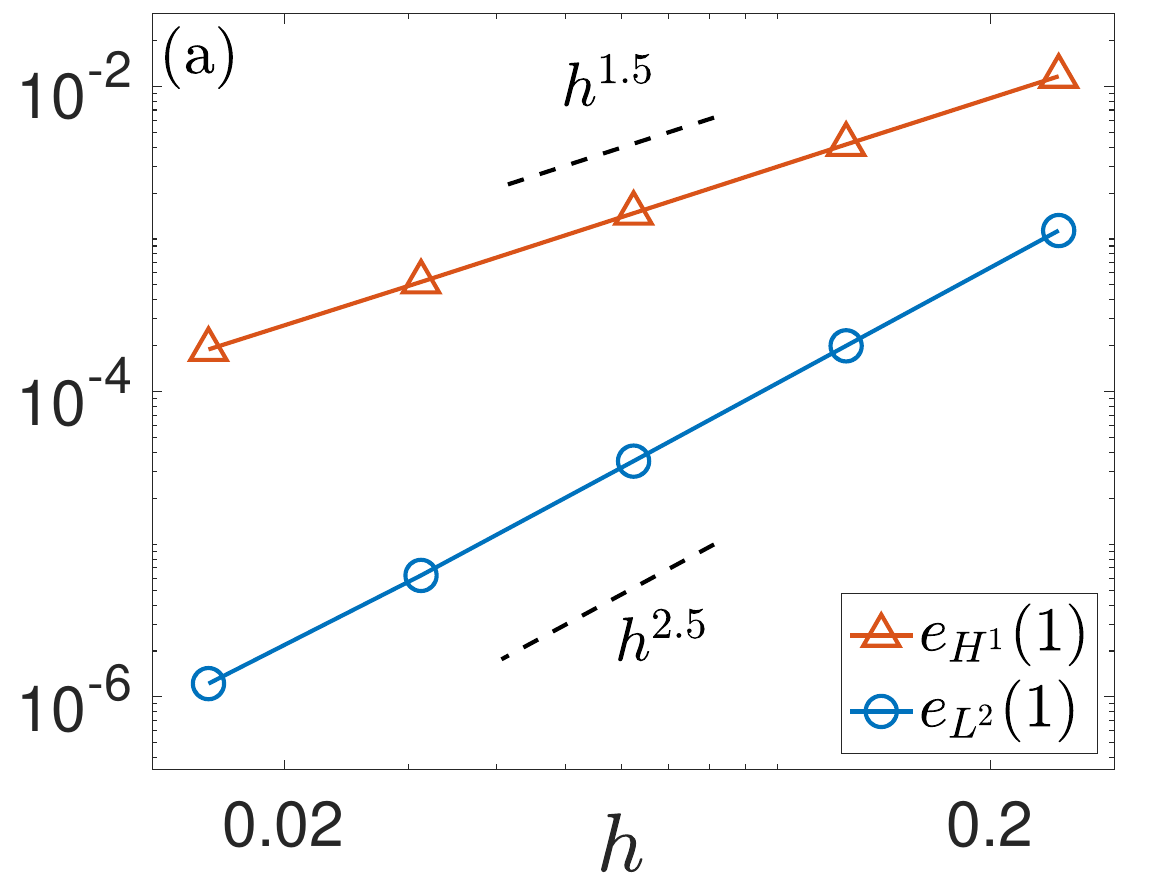}}\hspace{1em}
		{\includegraphics[width=0.475\textwidth]{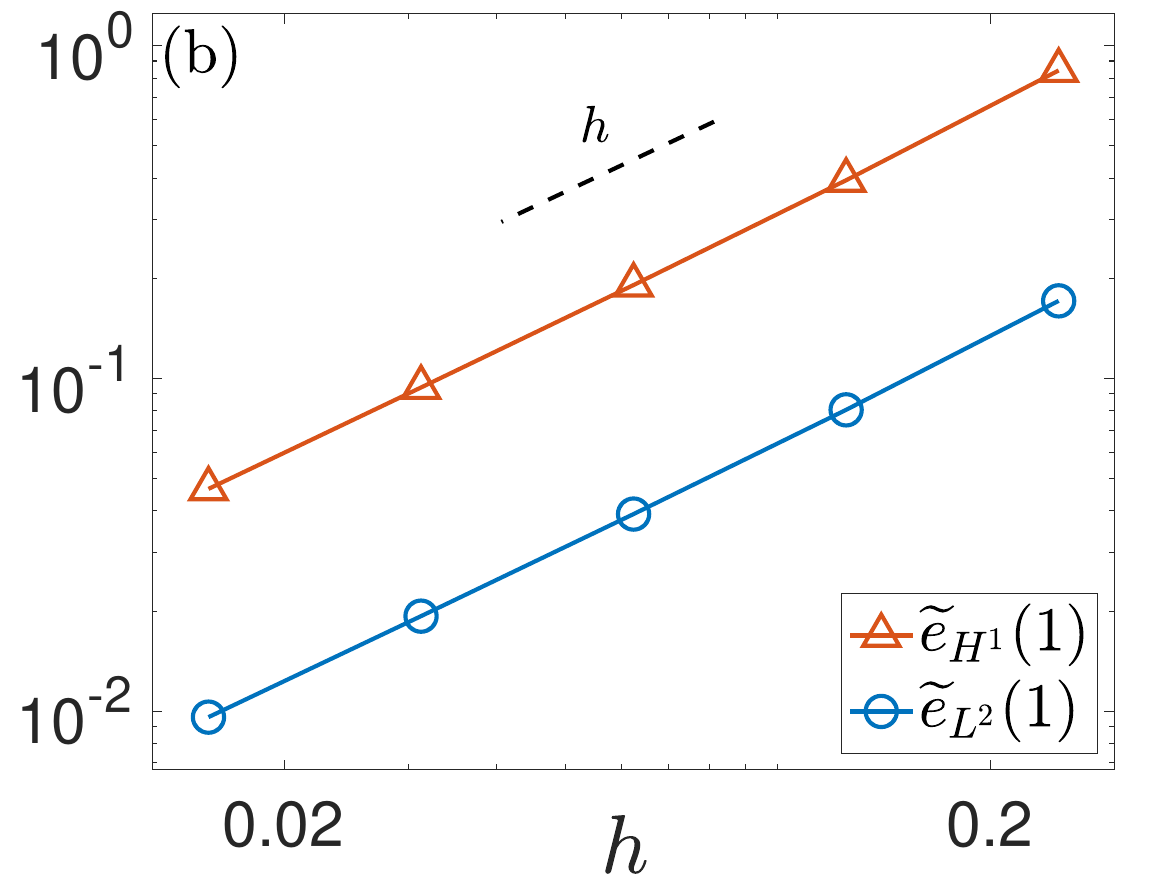}}
		\caption{Spatial errors of (a) LTSeFP method  and (b) LTSFP method for the GPE \cref{NLSE} with low regularity potential \cref{eq:poten_0}}
		\label{fig:0_space}
	\end{figure}
	
	We then consider the space-time separable potential \cref{eq:poten_1}. Similarly, in \cref{fig:1_time,fig:1_space}, we show the temporal errors and spatial errors of the LTSeFP method as well as the comparison with the LTSFP method. Again, the markers ``$\ast$" in \cref{fig:1_time} (b) are put at data points satisfying $ \tau = h^2/\pi$. We see that
	\begin{itemize}
		\item The temporal convergence of the LTeFP method is first-order in $L^2$- and $H^1$-norms when $\tau \leq h^2/\pi$. There is severe order reduction of temporal convergence in $H^1$-norm when $\tau \gg h^2$. 
		
		\item The spatial convergence of the LTeFP method is 2.5-order in $L^2$-norm and 1.5-order in $H^1$-norm, which is optimal with respect to the regularity of the exact solution as discussed in the previous example. Similarly, the LTSFP method is only first-order convergent in space in both $L^2$- and $H^1$-norms.  
	\end{itemize}
	The numerical results also confirm our error estimates \cref{thm:main}. We also note that the $L^2$-norm errors in time seem not very sensitive to the choice of the mesh size, which is different from the time-independent case above. However, from the $H^1$-norm errors in \cref{fig:1_time} (b), it is still necessary to choose $\tau \leq h^2/\pi$ for a good approximation to the exact solution. 
	
	\begin{figure}[htbp]
		\centering
		{\includegraphics[width=0.475\textwidth]{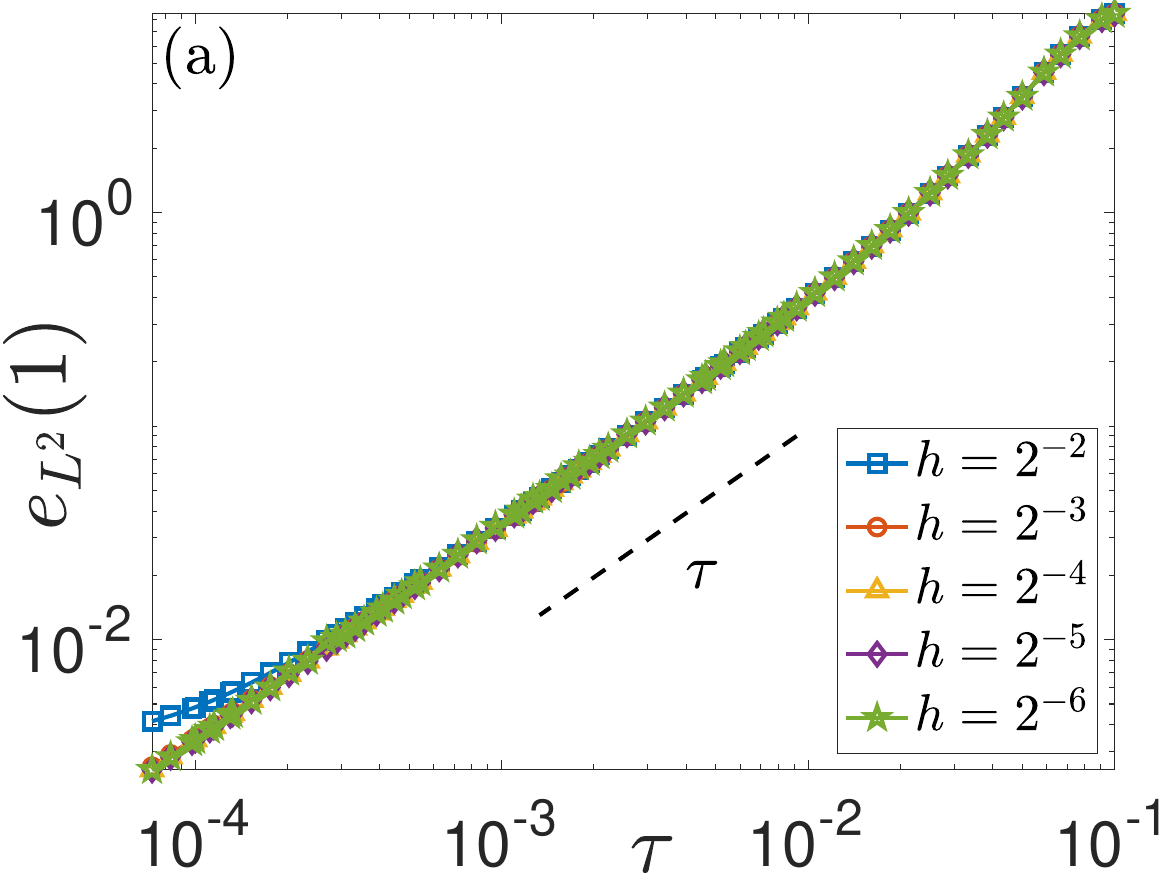}}\hspace{1em}
		{\includegraphics[width=0.475\textwidth]{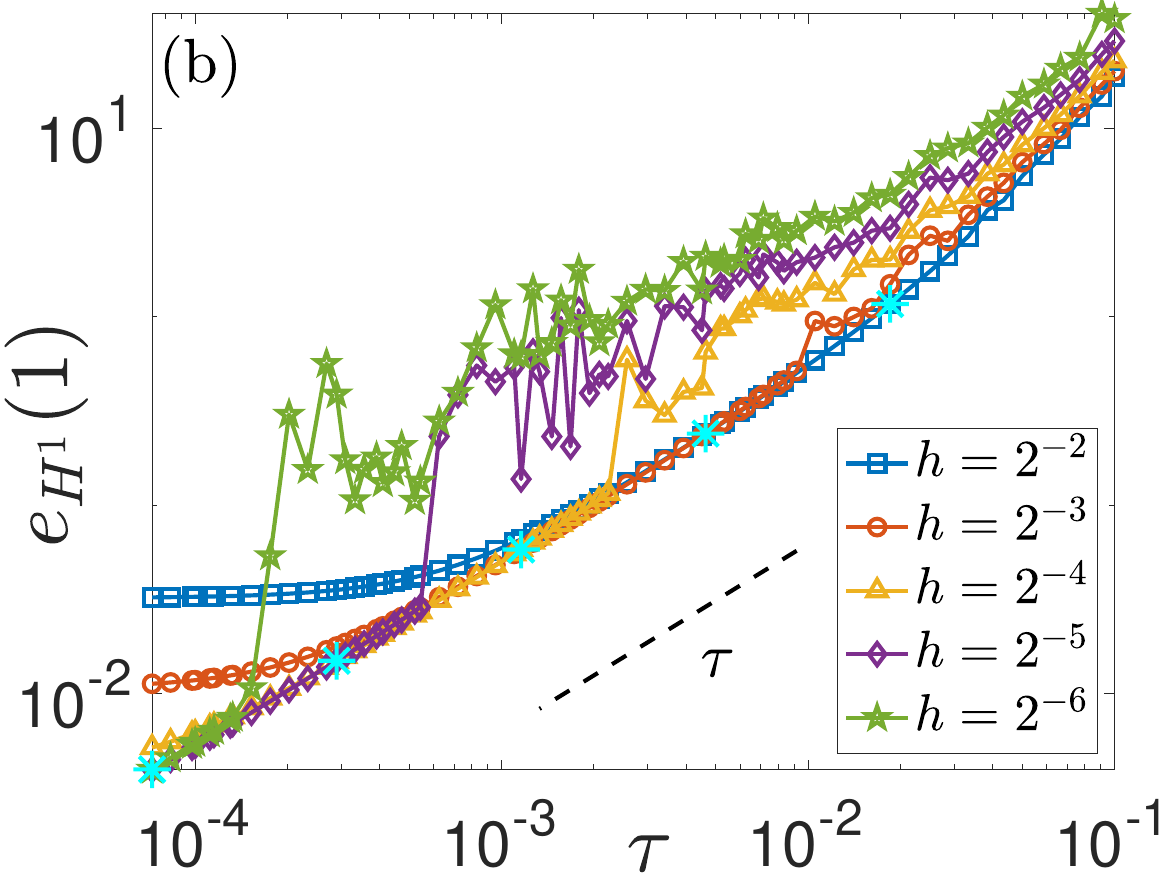}}
		\caption{Temporal errors of the LTSeFP method for the GPE \cref{NLSE} with low regularity potential \cref{eq:poten_1}: (a) $L^2$-norm errors and (b) $H^1$-norm errors }
		\label{fig:1_time}
	\end{figure}
	
	\begin{figure}[htbp]
		\centering
		{\includegraphics[width=0.475\textwidth]{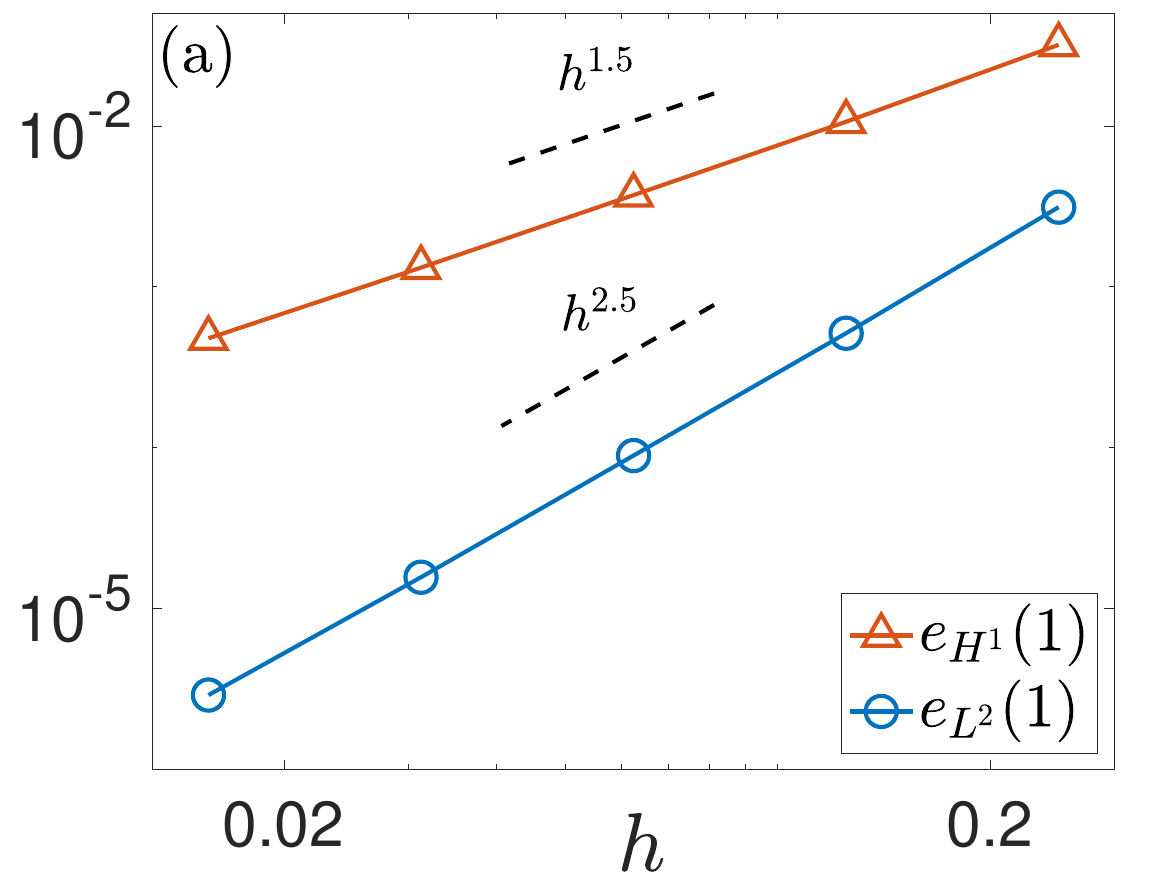}}\hspace{1em}
		{\includegraphics[width=0.475\textwidth]{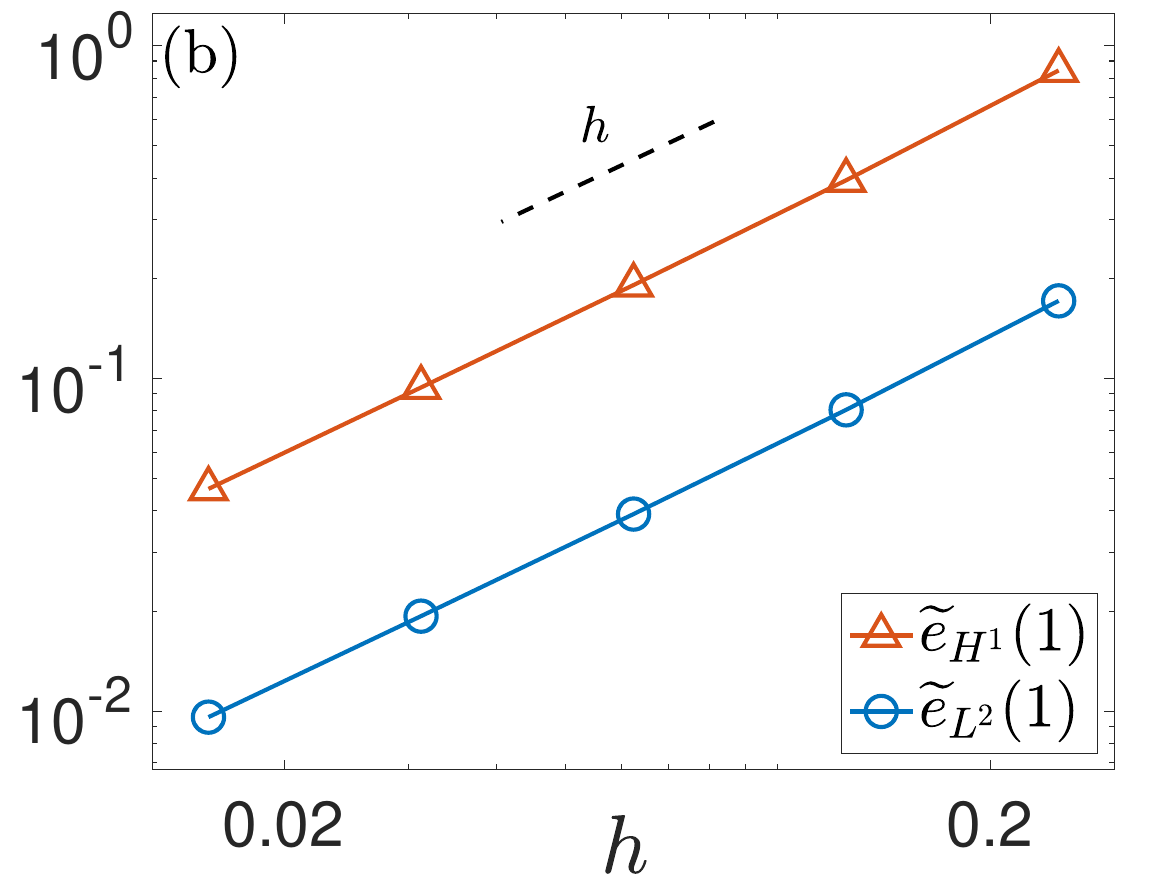}}
		\caption{Spatial errors of (a) LTSeFP method  and (b) LTSFP method for the GPE \cref{NLSE} with low regularity potential \cref{eq:poten_1}}
		\label{fig:1_space}
	\end{figure}
	
	Then we proceed to consider the time-dependent potential \cref{eq:poten_2}. We remark here that this potential does not satisfy the regularity assumptions made in \cref{thm:main} for the error estimates. Again, we plot the temporal errors with different choices of the mesh size $h$ in \cref{fig:2_time}. The spatial errors and the comparison with the LTSFP method are shown in \cref{fig:1_space}. Similar observations can be made as in the previous example with the same conclusions. The only difference is that the standard Fourier pseudospectral method shows an even lower convergence order in $H^1$-norm as can be seen from \cref{fig:2_space} (b).  
	
	\begin{figure}[htbp]
		\centering
		{\includegraphics[width=0.475\textwidth]{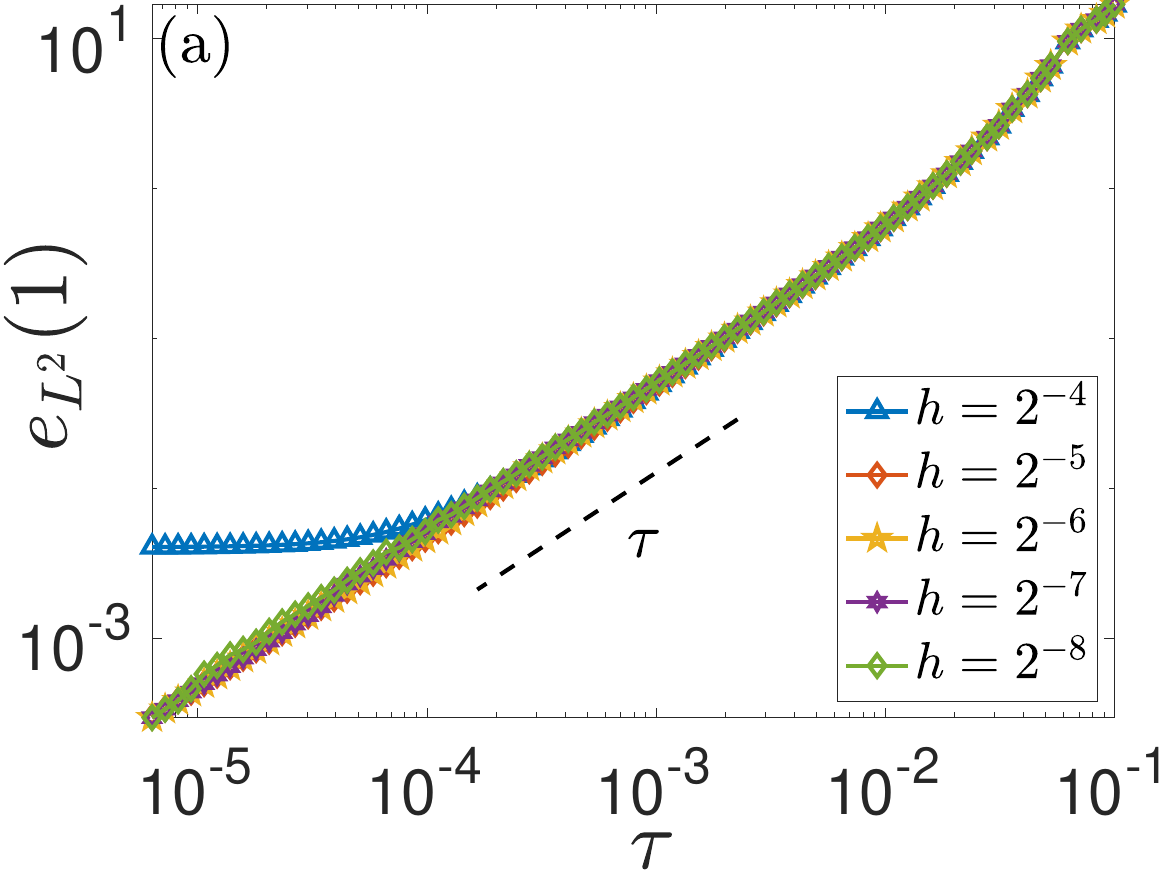}}\hspace{1em}
		{\includegraphics[width=0.475\textwidth]{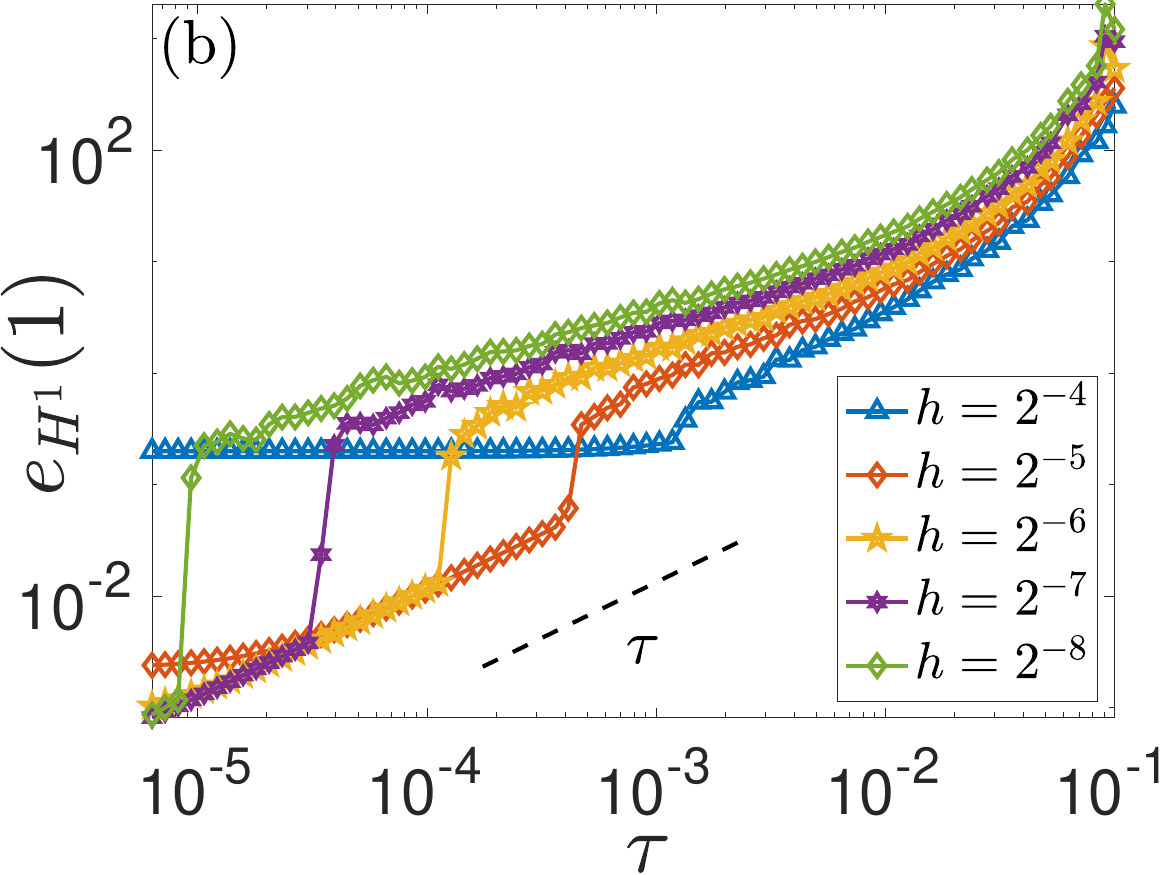}}
		\caption{Temporal errors of the LTSeFP method for the GPE \cref{NLSE} with low regularity potential \cref{eq:poten_2}: (a) $L^2$-norm errors and (b) $H^1$-norm errors }
		\label{fig:2_time}
	\end{figure}
	
	\begin{figure}[htbp]
		\centering
		{\includegraphics[width=0.475\textwidth]{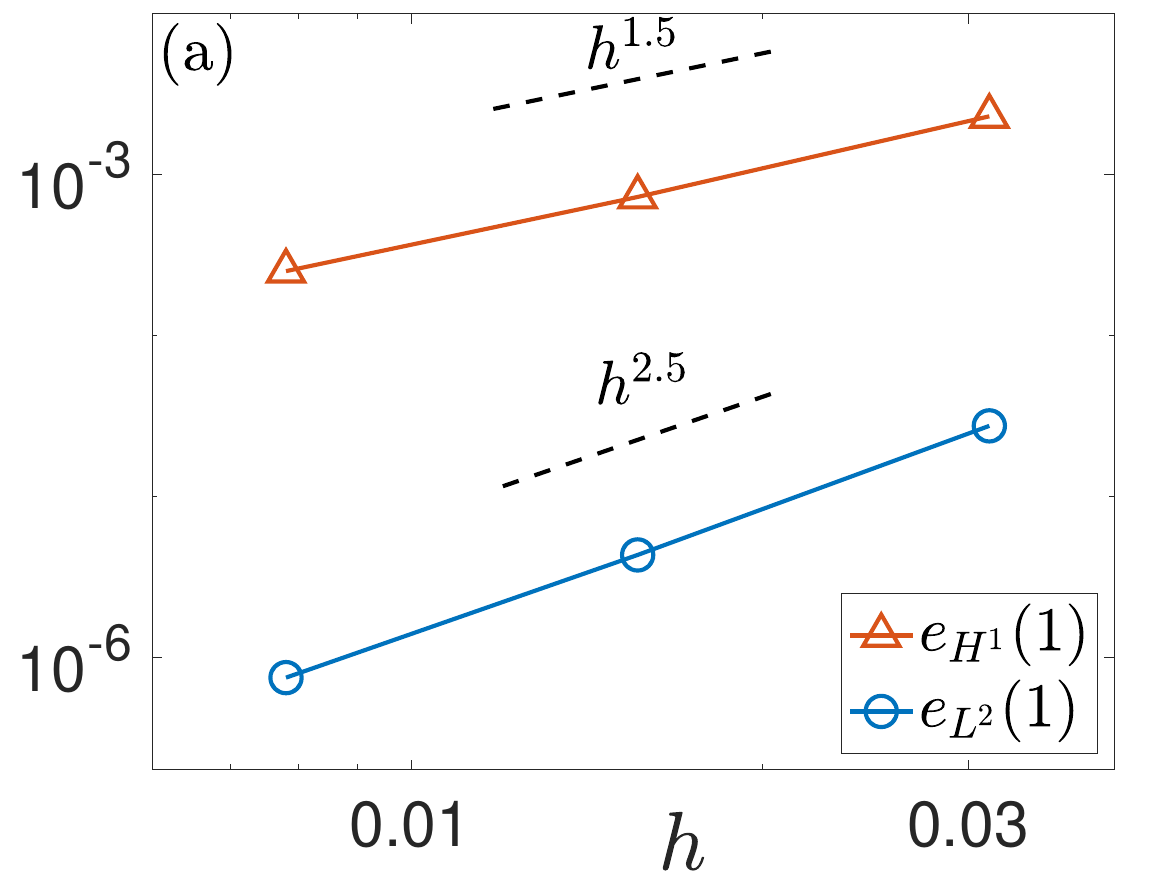}}\hspace{1em}
		{\includegraphics[width=0.475\textwidth]{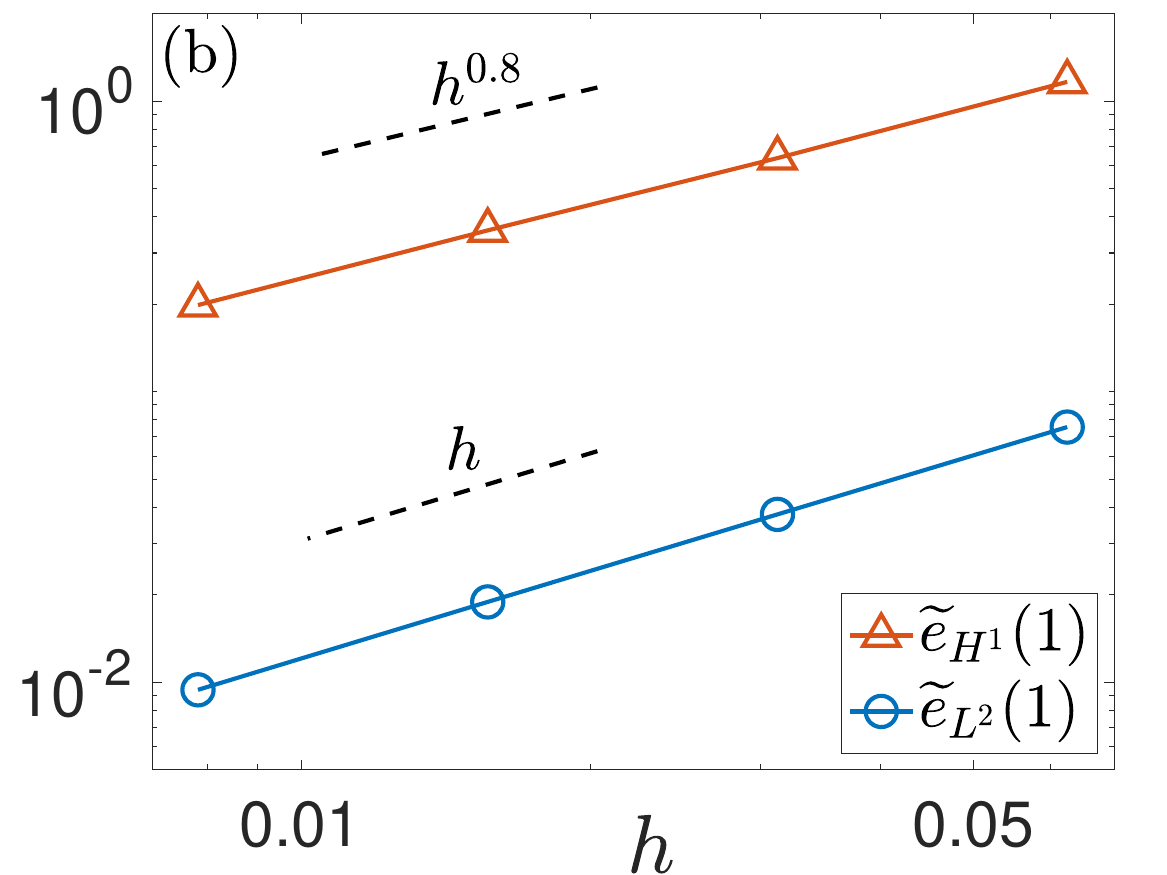}}
		\caption{Spatial errors of (a) LTSeFP method  and (b) LTSFP method for the GPE \cref{NLSE} with low regularity potential \cref{eq:poten_2}}
		\label{fig:2_space}
	\end{figure}
	
	\subsection{Applications}
	In this subsection, we shall use the LTSeFP method to compute some interesting phenomena including (i) the dynamics of cold atoms in a trapping potential continuously changing from single well to double well and (ii) a moving obstacle in superfluid. 
	
	For case (i), we choose $d=1$, $ \Omega = [-16, 16] $, $T_\text{max} = 1$ and use a square-well trapping potential
	\begin{equation}
		V_\text{trap}(x) = \left\{
		\begin{aligned}
			&0, &x \in (-6, 6), \\
			&x^2/2, &\text{otherwise}, 
		\end{aligned}
		\right. , 
		\qquad x \in \Omega. 
	\end{equation}
	In the middle of the well, we add a time-dependent potential barrier with width $l$ as
	\begin{equation}
		V_\text{bar}(x, t) = \left\{
		\begin{aligned}
			&5t, &x \in (-l, l), \\
			&0, &\text{otherwise}, 
		\end{aligned}
		\right. , 
		\qquad x \in \Omega. 
	\end{equation}
	to continuously change the single well into a double well. 
	
	In the simulation, we set $\beta = 1$, $V(x, t) = V_\text{trap}(x) + V_\text{bar}(x, t)$ in \cref{NLSE} and choose $\tau = 10^{-5}$, $h = 2^{-7}$. The initial datum is chosen as the ground state of the GPE \cref{NLSE} with time-independent square-well potential $V = V_\text{trap}$, which is computed by the gradient flow with discrete normalization (GFDN), also known as imaginary time evolution method \cite{bao2004,liu2021}. 
	
	In \cref{fig:1_dym}, we plot the density $|\psi(\cdot, t)|^2$ at different time $t$ and compare the influence of different widths of the potential barrier. We can observe that when the separating width of the barrier is small (\cref{fig:1_dym} (a) with $l=0.5$), the atoms are well separated into two wells while when the width is large (\cref{fig:1_dym} (b) with $ l=2 $), some atoms are also trapped in the barrier. 
	
	\begin{figure}[htbp]
		\centering
		{\includegraphics[width=0.475\textwidth]{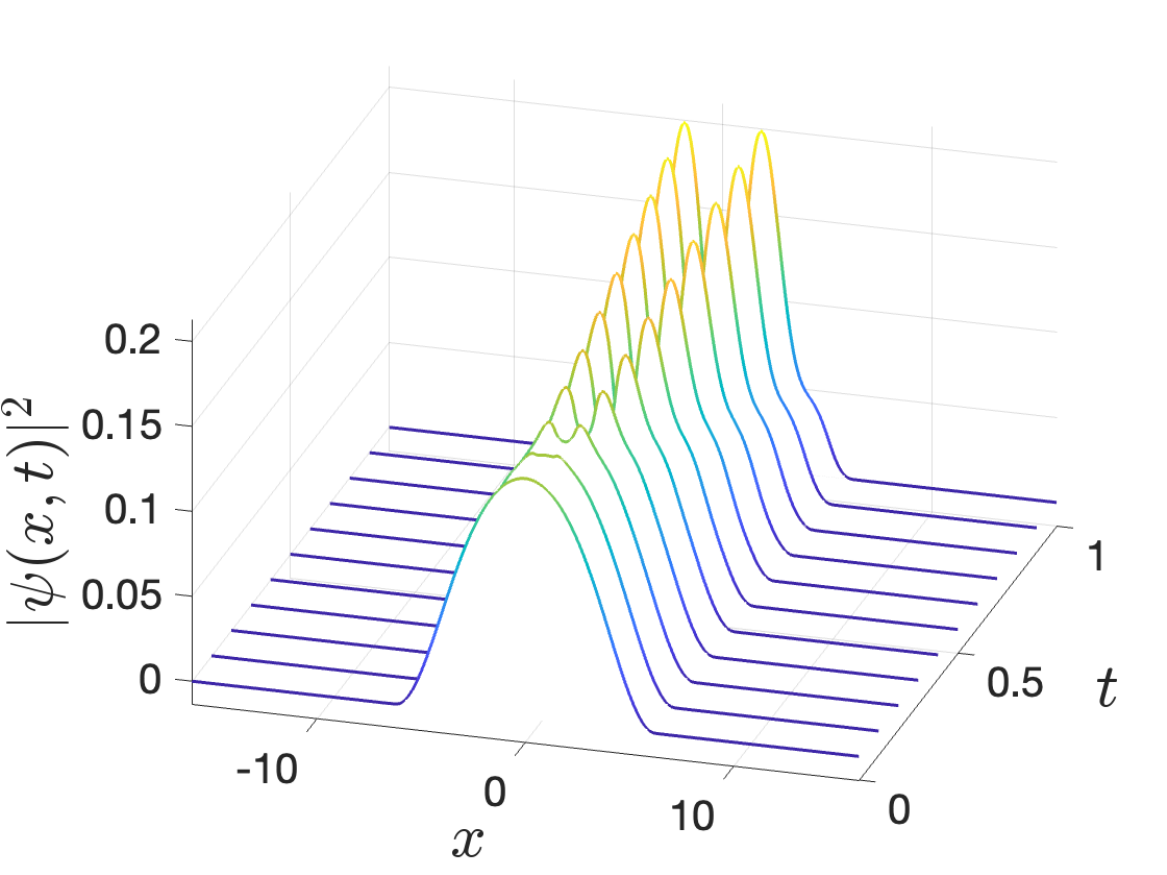}}\hspace{1em}
		{\includegraphics[width=0.475\textwidth]{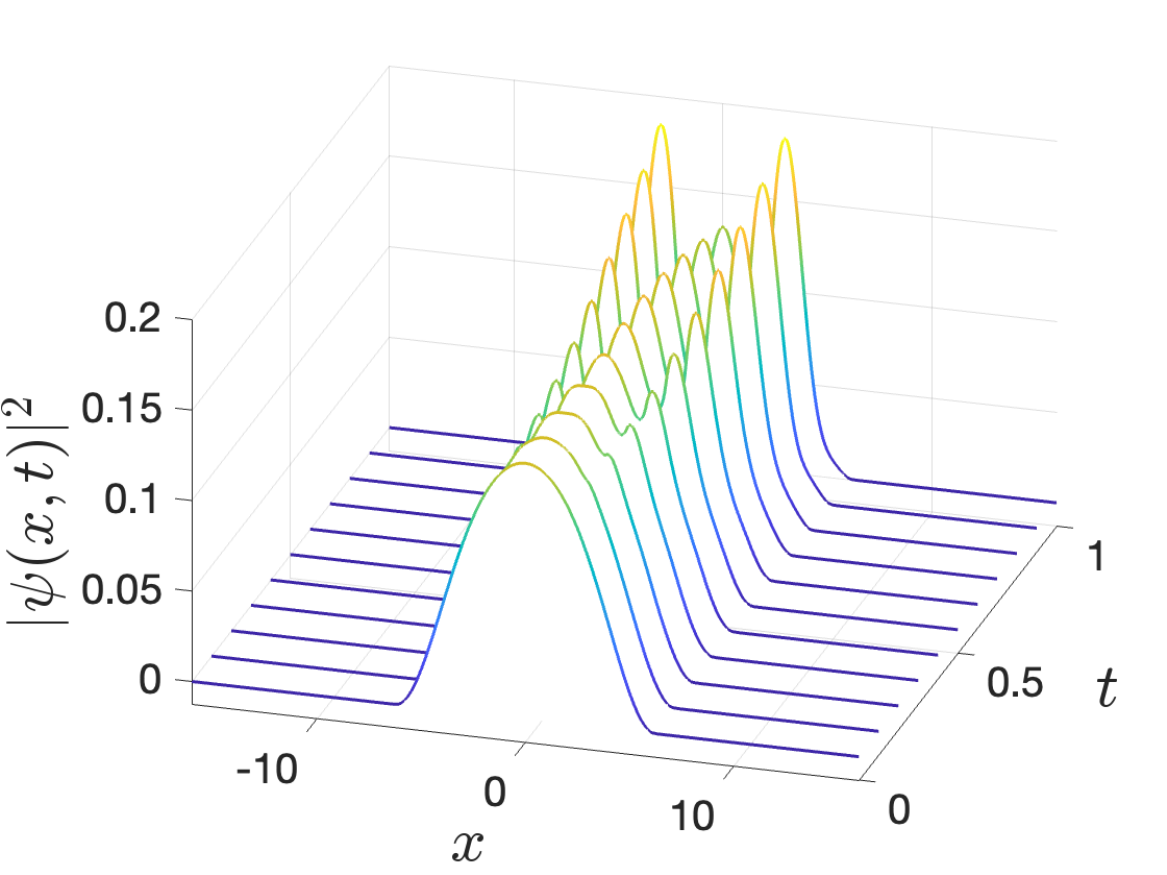}}
		\caption{Density at different time steps: (a) barrier width $l = 0.5$ and (b) barrier width $l = 2$}
		\label{fig:1_dym}
	\end{figure}
	
	For case (ii), we consider the 2D setting by choosing $d=2$ and $\Omega = (-8, 8) \times (-4, 4)$. The trapping potential is chosen as an anisotropic harmonic trap with strong confinement in the $y$-direction:
	\begin{equation}
		V_\text{ho}(\vx) = x^2/2 + 8 y^2, \qquad \vx = (x, y)^T \in \Omega. 
	\end{equation}
	The obstacle is modeled by a rectangular prism with height $A_0 = 10$, initially located at $\vx = \vx_0 = (5, 0)^T$ far away from the condensate and moving along $x$-direction with a constant speed $v$: 
	\begin{equation}
		V_\text{obs}(\vx, t) = A_0 V_\text{s}(\vx - \vx_0 + v t (1, 0)^T), \quad V_\text{s} (\vx) = \mathbbm{1}_{[-1/8, 1/8]^2} (\vx), \quad \vx = (x, y)^T \in \Omega.
	\end{equation}
	This setting is also depicted in \cref{fig:2_dym_ini} for a better illustration. 
	
	\begin{figure}[htbp]
		\centering
		{\includegraphics[width=\textwidth]{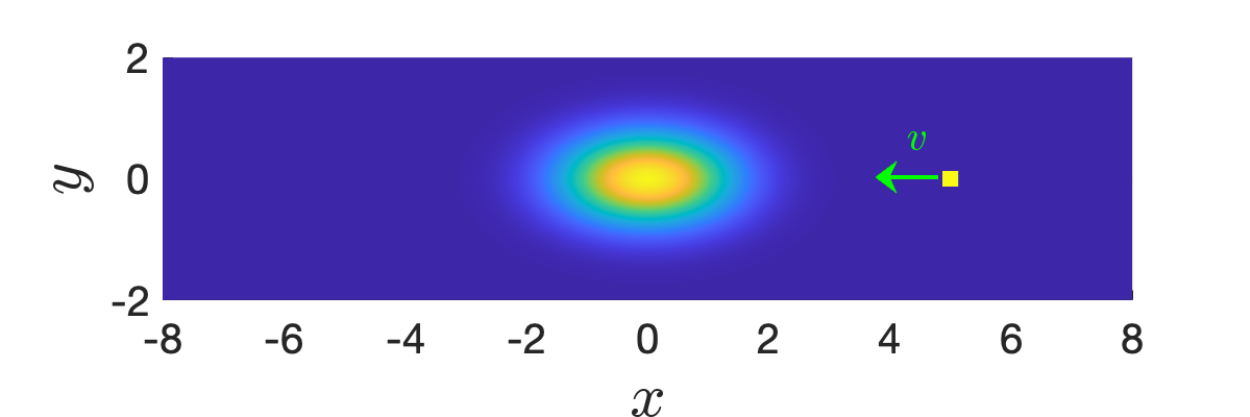}}
		\caption{Initial set-up}
		\label{fig:2_dym_ini}
	\end{figure}
	
	In the simulation, we set $V(\vx, t) = V_\text{ho}(\vx) + V_\text{obs}(\vx, t)$ and $\beta = 5$ in \cref{NLSE}. The time step size is chosen as $\tau = 2.5 \times 10^{-4}$ and the mesh size is $h=2^{-5}$ in both $x$- and $y$-directions. The initial data is chosen as a ground state of the GPE \cref{NLSE} with time-independent harmonic oscillator $V = V_\text{ho}$, which is computed again by the GFDN. 
	
	In \cref{fig:2_dym_1,fig:2_dym_2}, we plot the density $|\psi(\cdot, t)|^2$ at different time $t$ for the obstacle moving with velocity $v = 10$ and $v = 40$, respectively. We can observe that when the obstacle is moving with low velocity (\cref{fig:2_dym_1}), it can generate clear wave patterns with large wavelength, while, when the velocity is high (\cref{fig:2_dym_2}), the wave pattern is unclear and the wavelength is much smaller. 
	
	\begin{figure}[htbp]
		\centering
		{\includegraphics[width=\textwidth]{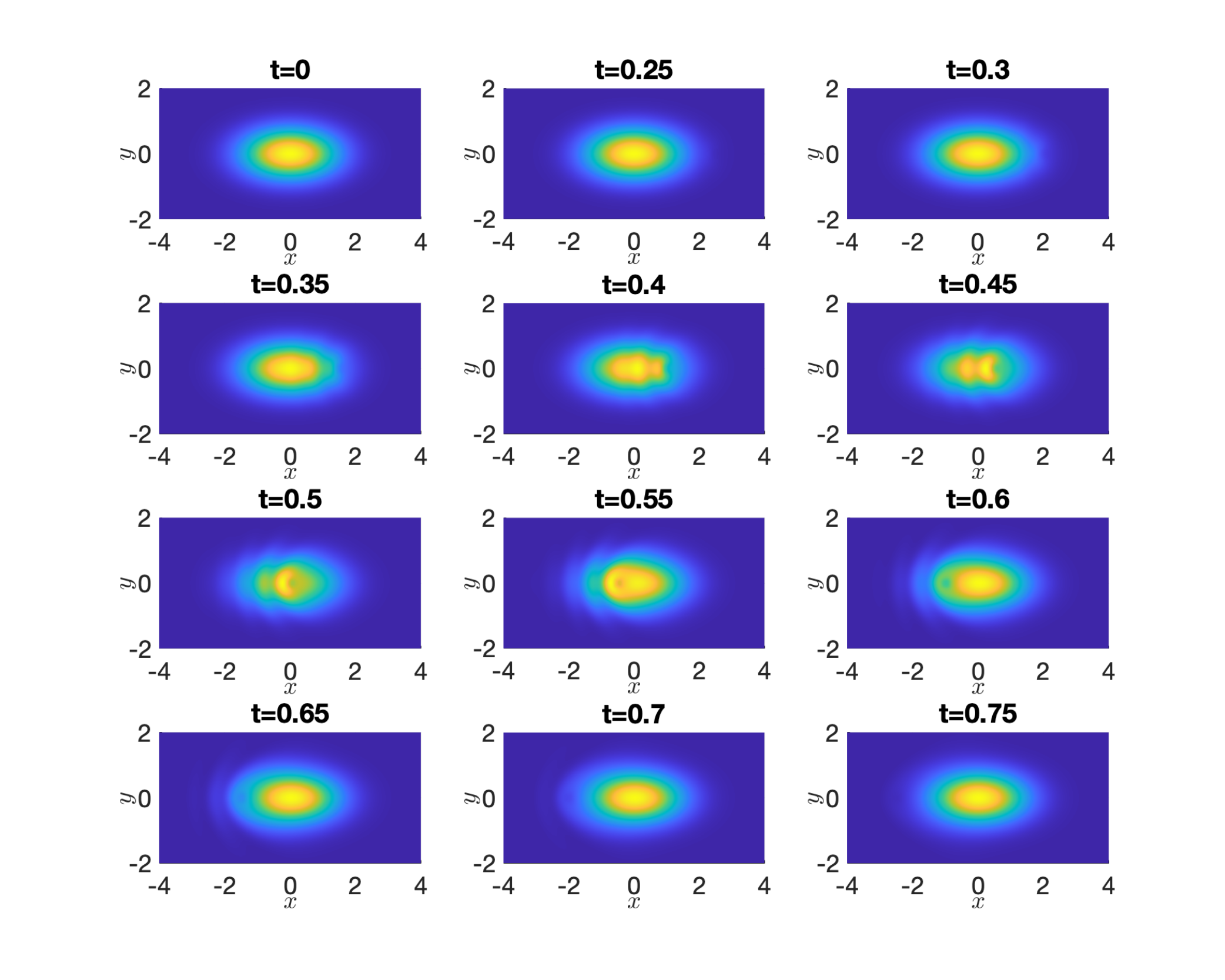}}
		\caption{Density at different time steps: obstacle velocity $v=10$ }
		\label{fig:2_dym_1}
	\end{figure}
	
	\begin{figure}[htbp]
		\centering
		{\includegraphics[width=\textwidth]{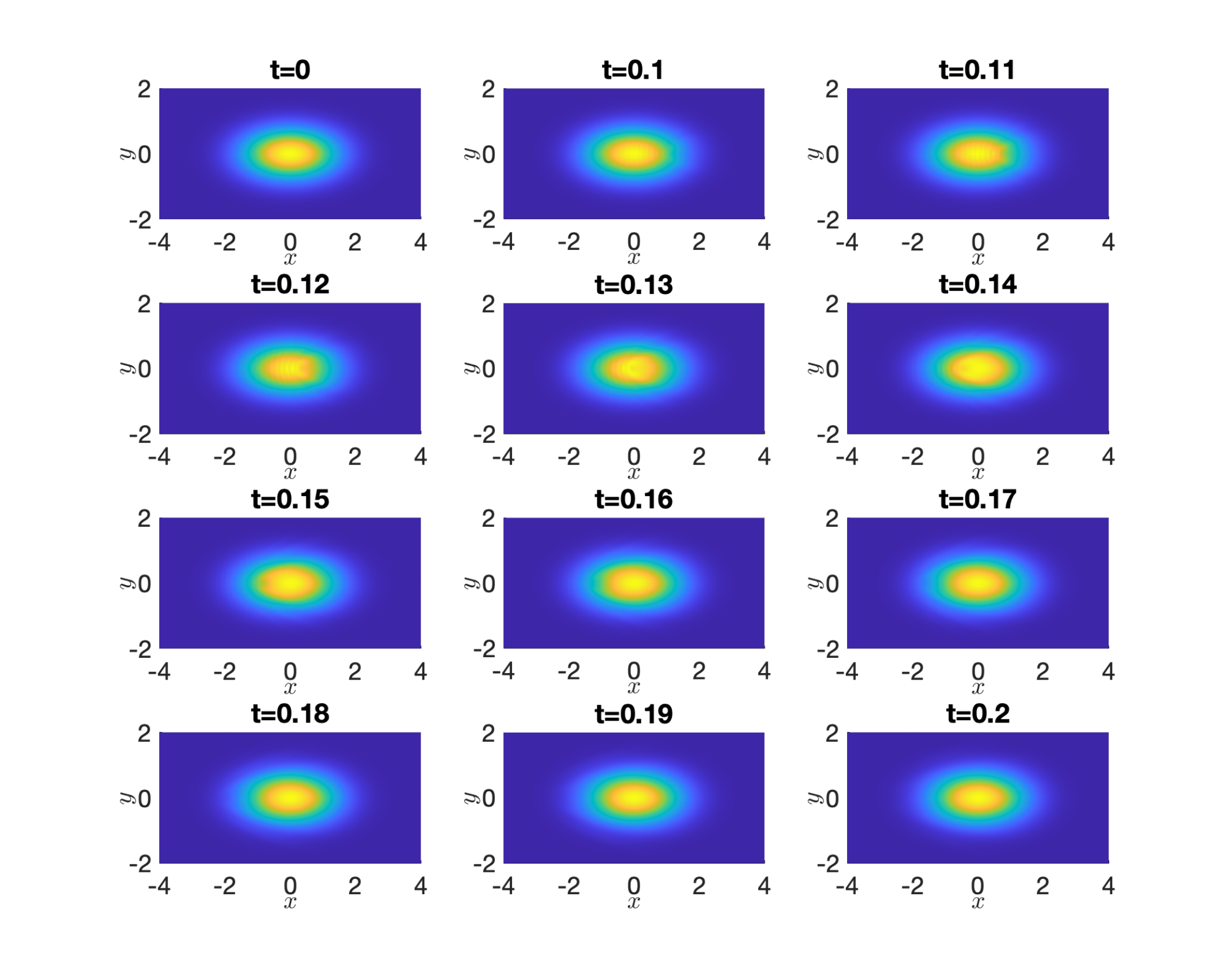}}
		\caption{Density at different time steps: obstacle velocity $v = 40$ }
		\label{fig:2_dym_2}
	\end{figure}
	
	\section{Conclusion}\label{sec:5}
	We proposed and analyzed a Lawson-time-splitting extended Fourier pseudospectral (LTSeFP) method for the Gross-Pitaevskii equation (GPE) with time-dependent low regularity potential. Theoreticaly, we proved that the LTSeFP method achieves first-order convergence in time and optimal-order convergence in space with respect to the regularity of the exact solution. Meanwhile, the computational cost of the LTSeFP method is comparable to the standard time-splitting Fourier pseudospectral method for most of physically relevant time-dependent low regularity potential. Hence, the LTSeFP is both accurate and efficient. Extensive numerical results were reported to confirm the error bounds and to show the excellent performance of the proposed method. 
	
	\section*{Acknowledgment}
	The work is partially supported by the Ministry of Education of Singapore under its AcRF Tier 2 funding MOE-T2EP20122-0002 (A-8000962-00-00). The authors would like to thank Prof. Weizhu Bao at National University of Singapore for fruitful discussions and instructions. 
	
	
	
	

\end{document}